\documentclass[11pt,reqno]{amsart}
\usepackage{amssymb}
\usepackage{amsmath, amssymb, amsthm}
\usepackage{mathrsfs}

\newtheorem{theorem}{Theorem}[section]

\newtheorem{lemma}{Lemma}[section]
\newtheorem{remark}{Remark}[section]

\numberwithin{equation}{section}

\begin{document}
\title[Magnetohydrodynamic equations]{On classical solutions of   the  compressible Magnetohydrodynamic equations with vacuum}

\author{Shengguo Zhu*}
\address[S. G. Zhu]{Department of Mathematics, Shanghai Jiao Tong University,
Shanghai 200240, P.R.China; School of Mathematics, Georgia Tech
Atlanta 30332, U.S.A.}
\email{\tt szhu44@math.gatech.edu}

\thanks{* correspondence author}

\begin{abstract}In this paper, we consider the  3-D compressible isentropic  MHD equations with infinity electric conductivity. The  existence of unique local classical solutions is  established  when the initial data is  arbitrarily large, contains vacuum and satisfies some initial layer compatibility condition. The initial mass density  needs not be bounded away from zero  and may vanish in some open set or decay at infinity. Moreover,  we  prove  that the $L^\infty$ norm of the deformation tensor of velocity gradients controls the possible blow-up  (see \cite{olga}\cite{zx}) for  classical (or strong) solutions, which  means that if a solution of the compressible MHD equations is initially regular and loses its regularity at some later time, then the formation of singularity must be caused  by  losing  the bound of the deformation tensor as the critical time approaches. Our criterion  (see (\ref{eq:2.91})) is the same as Ponce's criterion for $3$-D incompressible Euler equations \cite{pc} and Huang-Li-Xin's  criterion for the $3$-D compressible Navier-stokes equations \cite{hup}.
\end{abstract}

\date{Jan. 12, 2014}
\keywords{MHD, classical solutions, vacuum, compatibility condition, blow-up criterion.\\
{\bf Acknowledgments. } Shengguo Zhu's research was supported in part
by Chinese National Natural Science Foundation under grant 11231006 and 10971135.
}

\maketitle

\section{Introduction}
Magnetohydrodynamics is that part of the mechanics of continuous media which studies the motion of electrically conducting media in the presence of a magnetic field. The dynamic motion of  fluid and  magnetic field interact strongly on each other, so the hydrodynamic and electrodynamic effects are coupled. The applications of magnetohydrodynamics cover a very wide range of physical objects, from liquid metals to cosmic plasmas, for example, the intensely heated and ionized fluids in an electromagnetic field in astrophysics, geophysics, high-speed aerodynamics, and plasma physics.
In $3$-D space, the compressible isentropic   magnetohydrodynamic equations in a domain $\Omega $ of $\mathbb{R}^3$  can be written as
\begin{equation}
\label{eq:1.2j}
\begin{cases}
\displaystyle
H_t-\text{rot}(u\times H)=-\text{rot}\Big(\frac{1}{\sigma}\text{rot}H\Big),\\[6pt]
\displaystyle
\text{div}H=0,\\[6pt]
\displaystyle
\rho_t+\text{div}(\rho u)=0,\\[6pt]
\displaystyle
(\rho u)_t+\text{div}(\rho u\otimes u)
  +\nabla P=\text{div}\mathbb{T}+\mu_0\text{rot}H\times H.
\end{cases}
\end{equation}

In this system,  $x\in \Omega$ is the spatial coordinate; $t\geq 0$ is the time;  $H=(H^{(1)},H^{(2)},H^{(3)})$ is the magnetic field; $0< \sigma\leq \infty$ is the electric conductivity coefficient;  $\rho$ is the mass density;
$u=(u^{(1)},u^{(2)},u^{(3)})\in \Omega$ is the velocity of fluids; $P$ is the pressure law  satisfying 
\begin{equation}
\label{eq:1.3}
P=A \rho^\gamma, \quad \gamma > 1,
\end{equation}
where $A$  is a  positive constant  and $\gamma$ is the adiabatic index;
$\mathbb{T}$ is the stress tensor given by
\begin{equation}
\label{eq:1.4}
\mathbb{T}=2\mu D(u)+\lambda \text{div}u \mathbb{I}_3, \quad D(u)=\frac{\nabla u+(\nabla u)^\top}{2},
\end{equation}
where $D(u)$ is the deformation tensor,  $\mathbb{I}_3$ is the $3\times 3$ unit matrix, $\mu$ is the shear viscosity coefficient, $\lambda$ is the bulk viscosity coefficient, $\mu$ and $\lambda$ are both real constants,
\begin{equation}
\label{eq:1.5}
  \mu > 0, \quad \lambda+\frac{2}{3}\mu \geq 0,\end{equation}
which ensures the ellipticity of the   Lam$\acute{\text{e}}$ operator.
Although the electric field $E$ doesn't appear in  system (\ref{eq:1.2j}),
it is indeed induced according to a relation $
E=-\mu_0 u\times H
$
by moving the conductive flow in the magnetic field.

However, in this paper,  when $\sigma=+\infty$,   system (\ref{eq:1.2j}) can be written into
\begin{equation}
\label{eq:1.2}
\begin{cases}
\displaystyle
H_t-\text{rot}(u\times H)=0,\\[6pt]
\displaystyle
\text{div}H=0,\\[6pt]
\displaystyle
\rho_t+\text{div}(\rho u)=0,\\[6pt]
\displaystyle
(\rho u)_t+\text{div}(\rho u\otimes u)
  +\nabla P=\text{div}\mathbb{T}+\mu_0\text{rot}H\times H
\end{cases}
\end{equation}
with initial-boundary  conditions
\begin{align}
\label{fanmei}(H,\rho, u)|_{t=0}=(H_0(x), \rho_0(x),  u_0(x)),\quad   x\in \Omega,\quad  u|_{\partial {\Omega}}=0, \\
\label{fan1}(H(t,x),\rho(t,x), u(t,x),P(t,x)) \rightarrow (0,\overline{\rho},0,\overline{P}) \quad \text{as} \quad |x|\rightarrow \infty,\quad t>0,
\end{align}
where $\overline{\rho} \geq 0$  and $ \overline{P}=A\overline{\rho}^\gamma$ are both  constants, and  $\Omega$ can be a  bounded domain in $\mathbb{R}^3$ with smooth boundary  or the whole space $\mathbb{R}^3$.
We have to point out that,  if $\Omega$ is a bounded domain (or $\mathbb{R}^3$), then the condition (\ref{fan1}) at infinity  (or the boundary condition in (\ref{fanmei})  respectively)   should be neglected.

Throughout this paper, we adopt the following simplified notations for the standard homogeneous and inhomogeneous Sobolev space:
\begin{equation*}\begin{split}
&D^{k,r}=\{f\in L^1_{loc}(\Omega): |f|_{D^{k,r}}=|\nabla^kf|_{L^r}<+\infty\},\quad D^k=D^{k,2}, \\[6pt]
&D^{1}_0=\{f\in L^6(\Omega): |f|_{D^{1}}=|\nabla f|_{L^2}<\infty \ \text{and}\ f|_{\partial \Omega}=0\},\quad \|(f,g)\|_X=\|f\|_X+\|g\|_X,\\[6pt]
 &\|f\|_s=\|f\|_{H^s(\Omega)},\quad |f|_p=\|f\|_{L^p(\Omega)},\quad |f|_{D^k}=\|f\|_{D^k(\Omega)}, \ \mathbb{A}: \mathbb{B}=(a_{ij}b_{ij})_{3\times 3},\\[2pt]
&f\cdot \nabla g=\sum_{i=1}^{3}f_i \partial_i g,\quad f\cdot (\nabla g)=(\sum_{i=1}^3 f_i \partial_1 g_i,\sum_{i=1}^3 f_i \partial_2 g_i,\sum_{i=1}^3 f_i \partial_3 g_i)^\top,
\end{split}
\end{equation*}
where $f=(f_1,f_2,f_3)^\top\in \mathbb{R}^3$ or $f\in \mathbb{R}$, $g=(g_1,g_2,g_3)^\top\in \mathbb{R}^3$   or $g \in \mathbb{R}$,  $X$ is some Sobolev space, $\mathbb{A}=(a_{ij})_{3\times 3}$ and $\mathbb{B}=(b_{ij})_{3\times 3}$ are both $3\times 3$ matrixes. A detailed study of homogeneous Sobolev space may be found in \cite{gandi}.

As has been observed in \cite{jishan},  which proved the existence of unique local strong solution with initial vacuum,  in order to make sure that the Cauchy problem or IBVP (\ref{eq:1.2})-(\ref{fan1}) with  vacuum is well-posed, the lack of a positive lower bound of the initial mass density $\rho_0$ should be compensated with some initial layer compatibility condition on the initial data $(H_0,\rho_0, u_0,P_0)$.  For classical solution,  it can be shown as
\begin{theorem}\label{th5}Let constant $q \in (3,6]$.
If the initial data $(H_0, \rho_0,  u_0,P_0)$ satisfies
\begin{equation}\label{th78}
\begin{split}
& (H_0, \rho_0-\overline{\rho}, P_0-\overline{P}) \in H^2\cap W^{2,q},\  \rho_0\geq 0,\    u_0\in D^1_0\cap D^2,
\end{split}
\end{equation}
and the compatibility condition
\begin{equation}\label{th79}
\begin{split}
Lu_0+\nabla P_0- \text{rot} H_0\times H_0=\sqrt{\rho_0} g_1
\end{split}
\end{equation}
for some $g_1 \in L^2$, where $L$ is the Lam$\acute{\text{e}}$ operator defined via 
$$
Lu=-\mu\triangle u-(\lambda+\mu)\nabla  \text{div} u.
$$
Then there exists a small time $T_*$ and a unique solution $(H,\rho,u,P)$ to IBVP (\ref{eq:1.2})-(\ref{fan1}) satisfying
\begin{equation}\label{regs}\begin{split}
&(H,\rho-\overline{\rho},P-\overline{P})\in C([0,T_*];H^2\cap W^{2,q}),\\
&u\in C([0,T_*];D^1_0\cap D^2)\cap  L^2([0,T_*];D^3)\cap L^{p_0}([0,T_*];D^{3,q}),\ u_t\in  L^2([0,T_*];D^1_0),\\
&\sqrt{\rho}u_t\in L^\infty([0,T_*];L^2),\ t^{\frac{1}{2}}u\in L^\infty([0,T_*];D^3),\ t^{\frac{1}{2}}\sqrt{\rho}u_{tt}\in L^2([0,T_*];L^2),\\
& t^{\frac{1}{2}}u_t\in L^\infty([0,T_*];D^1_0)\cap L^2([0,T_*];D^2),\ tu\in L^\infty([0,T_*]; D^{3,q}),\\
&\ tu_t\in L^\infty([0,T_*]; D^{2}),\ tu_{tt}\in L^2([0,T_*];D^1_0),\ t\sqrt{\rho}u_{tt}\in L^\infty([0,T_*];L^2),
\end{split}
\end{equation}
where $p_0$ is a constant satisfying $1\leq p_0\leq \frac{4q}{5q-6}\in (1,2)$.

\end{theorem}

\begin{remark}\label{re1}
The solution we obtained in Theorem \ref{th5}  becomes a classical one for positive time. Some similar results have been obtained in \cite{jishan}\cite{dehua}, which give the local existence of strong solutions.  So, the main purpose of this theorem is to give a better regularity for the solutions obtained in \cite{jishan}\cite{dehua} when the initial mass density is nonnegative.
\end{remark}

Though the smooth global solution near the constant state in one-dimensional case has been  studied in \cite{kawa},  however, in $3$-D space, the non-global existence has been proved for the classical solution   to isentropic magnetohydrodynamic equations in \cite{olga} as follows:
\begin{theorem}\cite{olga} \label{olga}
Assume that $\gamma\geq \frac{6}{5}$, if the momentum $\int_{\Omega} \rho u \text{d} x\neq 0 $, then there exists no global classical solution to (\ref{eq:1.2})-(\ref{fan1}) with conserved mass and total energy.
\end{theorem}

So, naturally,  we want to understand the mechanism of blow-up and the structure of possible singularities: what kinds of singularities will form in finite time and what is the main mechanism of possible breakdown of smooth solutions for the $3$-D compressible MHD equations? 
Therefore, it is an interesting question to ask whether the same blow-up criterion in terms of $D(u)$ in \cite{hup}\cite{pc} still holds for the compressible MHD equations or not. However, the similar result has been obtained in Xu-Zhang  \cite{gerui} for  strong solutions obtained in \cite{jishan},  which is in terms of $\nabla u$:
\begin{equation}\label{eq:2.911}
\lim \sup_{ T \rightarrow \overline{T}} |\nabla u|_{L^1([0,T]; L^\infty(\Omega))}=\infty.
\end{equation}
 Based on  a subtle estimate for the magnetic field, our main result in this paper  answered this question for  classical  (or strong) solutions   positively, which  can be shown as

\begin{theorem} [\textbf{ Blow-up criterion for the IBVP  (\ref{eq:1.2})--(\ref{fan1})}]\label{th3}\ \\
Assume that  $\Omega $ is a bounded domain and the initial data  $(H_0,\rho_0,u_0, P_0)$ satisfies (\ref{th78})-(\ref{th79}).
 Let $(H,\rho,u,P)$  is a classical solution to IBVP for (\ref{eq:1.2})--(\ref{fan1}). 
If $0< \overline{T} <\infty$ is the maximal time of existence, then
\begin{equation}\label{eq:2.91}
\lim \sup_{ T \rightarrow \overline{T}} |D( u)|_{L^1([0,T]; L^\infty(\Omega))}=\infty.
\end{equation}
Moreover, our blow-up criterion also holds for the strong solutions obtained in \cite{jishan}.
\end{theorem}

\begin{remark}When $H\equiv 0$ in $3$-D space, the existence of unique local strong solution with vacuum has been solved by many papers, and  we refer the readers to \cite{CK3}\cite{CK}\cite{guahu}. Huang-Li-Xin obtained the  well-posedness of global classical solutions with small energy but possibly large oscillations and vacuum for Cauchy problem \cite{HX1} or IBVP \cite{HX2}.

However,  for compressible non-isentropic Navier-Stokes equations,    the finite time blow-up has been proved in  Olga \cite{olg1} for classical solutions $(\rho, u,S)$ $(S\  is\  the\  entropy)$ with  highly decreasing at infinity for the compressible non-isentropic Navier-stokes equations, but the local existence for the corresponding smooth solution is still open. 

 Recently,  Xin-Yan \cite{xy}  showed that if the initial vacuum only appears in some local domain, the smooth solution $(\rho,  \theta, u )$ to the Cauchy problem (\ref{eq:1.2})--(\ref{fan1}) will blow-up in finite time regardless of the size of initial data, which has removed the key assumption that the vacuum must appear in the far field in \cite{zx}.

Sun-Wang-Zhang \cite{zif}\cite{zhang2d} established a Beal-Kato-Majda blow-up criterion in terms of the upper bound of density for the strong solution with vacuum in $3$-D or $2$-D  space, which is weaker than the blow-up criterions obtained in \cite{hup}\cite{pc}.  Then our result can not replace $\int_0^{T} |D(u)|_\infty\text{d}t$ by 
$|\rho|_\infty$ because of the coupling of $u$ and $H$ in magnetic equation and  the lack of smooth mechanism of $H$.

Moreover,  results presented above are essentially dependent of the strong ellipticity of Lam$\acute{ \text{e}}$ operator.
 Compared with Euler equations \cite{tms1}, the velocity $u$ of fluids satisfies
$ Lu_0=0$ in the vacuum domain naturally due to the constant viscosity coefficients which makes sure that $u$ is well defined in the vacuum points without other assumptions as   \cite{tms1}.

 Recently, Li-Pan-Zhu  \cite{lipan} proved  the local existence of regular solutions for  the $2$-D Shallow water equations with $\mathbb{T}=\rho \nabla u$ when initial mass density decays to zero,  and the corresponding Beal-Kato-Majda blow-up criterion is also obtained. 

\end{remark}

The rest of  this paper is organized as follows. In Section $2$, we give some important lemmas which will be used frequently in our proof.  In Section $3$,   via  establishing a priori estimate (for the approximation solutions) which is independent of the lower bound of the initial mass density $\rho_0$, we can obtain the existence of unique local classical solution by the approximation process from non-vacuum to vacuum.  In Section $4$,  we give the proof for the blow-up criterion (\ref{eq:2.91}) for the classical solutions obtained in Section $3$. Firstly in Section $4.1$,  via assuming that the opposite of  (\ref{eq:2.91}) holds, we show that the solution in $ [0,\overline{T}]$ has the regularity that the strong solution has to satisfy obtained in \cite{jishan}. Then secondly  in Section $4.2$, based on the estimates shown in Section $4.1$, we improve the regularity of $(H,\rho,u,P)$ to make sure that  it is also  a classical one in $ [0,\overline{T}]$, which contradicts our assumption.
\section{Preliminary}

Now we give some important Lemmas which will be used frequently in our proof.

\begin{lemma}\cite{amj}\label{zhen1}
Let constants $l$, $a$ and $b$ satisfy the relation $\frac{1}{l}=\frac{1}{a}+\frac{1}{b}$ and  $1\leq a,\ b, \ l\leq \infty$. $ \forall s\geq 1$, if $f, g \in W^{s,a} \cap  W^{s,b}(\Omega)$, then we have
\begin{equation}\begin{split}\label{ku11}
&|D^s(fg)-f D^s g|_l\leq C_s\big(|\nabla f|_a |D^{s-1}g|_b+|D^s f|_b|g|_a\big),
\end{split}
\end{equation}
\begin{equation}\begin{split}\label{ku22}
&|D^s(fg)-f D^s g|_l\leq C_s\big(|\nabla f|_a |D^{s-1}g|_b+|D^s f|_a|g|_b\big),
\end{split}
\end{equation}
where $C_s> 0$ is a constant only depending on $s$.
\end{lemma}
The proof can be seen in Majda \cite{amj}, here we omit it.
The following one is some Sobolev inequalities obtained from the   well-known Gagliardo-Nirenberg inequality:

\begin{lemma}\label{gag}
For  $n\in (3,\infty)$, there exists some generic constant $C> 0$ that may depend  $n$ such that for $f\in D^1_0(\Omega)$,  $g\in D^1_0\cap D^2(\Omega)$  and $h \in W^{1,n}(\Omega)$, we have
\begin{equation}\label{gaga1}
\begin{split}
|f|_6\leq C|f|_{D^1_0},\qquad |g|_{\infty}\leq C|g|_{D^1_0\cap D^2}, \qquad |h|_{\infty}\leq C\|h\|_{W^{1,n}}.
\end{split}
\end{equation}
\end{lemma}
The next lemma is important in the derivation of our local a priori estimate for the higher order term of $u$, which can be seen in the Remark 1 of \cite{bjr}.
\begin{lemma}\label{bei}
If $h(t,x)\in L^2(0,T; L^2)$, then there exists a sequence $s_k$ such that
$$
s_k\rightarrow 0, \quad \text{and}\quad s_k |h(s_k,x)|^2_2\rightarrow 0, \quad \text{as} \quad k\rightarrow\infty.
$$
\end{lemma}

Based on  Harmonic analysis,  we introduce a regularity estimate result for Lam$\acute{ \text{e}}$ operator
\begin{equation}\label{ok}
-\mu\triangle u-(\mu+\lambda)\nabla \text{div}u =Lu=F \quad \text{in}\quad \Omega,  \quad  u\rightarrow 0 \quad \text{as} \ |x|\rightarrow \infty.
\end{equation}
We define $ u\in D^{1,q}_0(\Omega)$ means that $u \in D^{1,q}(\Omega)$ with $u|_{\partial \Omega}=0$.
\begin{lemma}\cite{harmo}\label{zhenok}
Let  $u\in D^{1,l}_0$   with $1< l< \infty$ be a weak solution to system (\ref{ok}),  if $\Omega= \mathbb{R}^3$, we have
$$
|u|_{D^{k+2,l}(\mathbb{R}^3)} \leq C |F|_{D^{k,l}(\mathbb{R}^3)};
$$
if $\Omega$  is a bounded domain with smooth boundary, we have
$$
|u|_{D^{k+2,l}(\Omega)} \leq C\big( |F|_{D^{k,l}(\Omega)}+|u|_{D^{1,l}_0(\Omega)}\big),
$$
where the constant   $C$ depending only on $\mu$,  $\lambda$  and $l$.
\end{lemma}
\begin{proof}
The  proof can be obtained via the classical estimates from Harmonic analysis, which can be seen in \cite{CK3}  \cite{harmo} or \cite{zif}.
\end{proof}

We also show some results obtained via the Aubin-Lions Lemma.
\begin{lemma}\cite{jm}\label{aubin} Let $X_0$, $X$ and $X_1$ be three Banach spaces with $X_0\subset X\subset X_1$. Suppose that $X_0$ is compactly embedded in $X$ and that $X$ is continuously embedded in $X_1$.\\[1pt]

I) Let $G$ be bounded in $L^p(0,T;X_0)$ where $1\leq p < \infty$, and $\frac{\partial G}{\partial t}$ be bounded in $L^1(0,T;X_1)$. Then $G$ is relatively compact in $L^p(0,T;X)$.\\[1pt]

II) Let $F$ be bounded in $L^\infty(0,T;X_0)$  and $\frac{\partial F}{\partial t}$ be bounded in $L^l(0,T;X_1)$ with $l>1$. Then $F$ is relatively compact in $C(0,T;X)$.
\end{lemma}

Finally, for $(H,u) \in C^2(\Omega)$,  there are some formulas based on  $\text{div}H=0$:
\begin{equation}\label{zhoumou}
\begin{cases}
\displaystyle
\text{rot}(u\times H)=(H\cdot \nabla)u-(u\cdot \nabla)H-H\text{div}u,\\[8pt]
\displaystyle
\text{rot}H\times H
=\text{div}\Big(H\otimes H-\frac{1}{2}|H|^2I_3\Big)=-\frac{1}{2}\nabla |H|^2+H\cdot \nabla H.
\end{cases}
\end{equation}

\section{Well-posedness of classical solutions}

 In order to prove
 the local existence of classical solutions to the original nonlinear problem, we need to consider the  following linearized problem:
\begin{equation}
\label{eq:1.2rfv}
\begin{cases}
\displaystyle
H_t+v\cdot \nabla H+(\text{div}vI_3-\nabla v)H=0\quad \text{in}\ (0,T)\times \Omega,\\[6pt]
\displaystyle
\text{div}H=0\quad \text{in}\ (0,T)\times \Omega,\\[6pt]
\displaystyle
\rho_t+\text{div}(\rho v)=0\quad \text{in}\ (0,T)\times \Omega,\\[6pt]
\displaystyle
\rho u_t+\rho v\cdot\nabla v
  +\nabla P+Lu=\mu_0\text{rot}H\times H\quad \text{in}\ (0,T)\times \Omega,\\[6pt]
(H,\rho,u)|_{t=0}=(H_0(x),\rho_0(x),u_0(x))\quad \text{in} \ \Omega, \\[6pt]
(H,\rho,u,P)\rightarrow (0,\overline{\rho},0,\overline{P}) \quad \text{as } \quad |x|\rightarrow \infty,\quad t> 0,
\end{cases}
\end{equation}
where  $(H_0(x),\rho_0(x),u_0(x))$ satisfies  (\ref{th78})-(\ref{th79})  and $v(t,x)\in \mathbb{R}^3$ is a known vector 
\begin{equation}\label{ghk1}
\begin{split}
&v\in C([0,T];D^1_0\cap D^2)\cap  L^2([0,T];D^3)\cap L^{p_0}([0,T];D^{3,q}),\ v_t\in  L^2([0,T];D^1_0),\\
&t^{\frac{1}{2}}v\in L^\infty([0,T];D^3),\ t^{\frac{1}{2}}v_t\in L^\infty([0,T];D^1_0) \cap L^2([0,T];D^2),\\
& tv\in L^\infty([0,T];D^{3,q}),\ tv_t\in L^\infty([0,T];D^2),\ tv_{tt}\in L^2([0,T];D^1_0),\ v(0,x)=u_0.
\end{split}
\end{equation}

\subsection{Unique solvability of (\ref{eq:1.2rfv}) away from  vacuum}\ \\

First we give  the following existence of  classical solution $(H,\rho,u)$ to (\ref{eq:1.2rfv})  by the standard methods at least for the case that the initial mass density is away from vacuum.

\begin{lemma}\label{lem1}
Assume in addition to (\ref{th78})-(\ref{th79}) that $\rho_0\geq \delta$ for some  constant $\delta>0$.
Then there exists a unique classical solution $(H,\rho,u)$ to (\ref{eq:1.2rfv}) such that
\begin{equation*}\begin{split}
&(H,\rho-\overline{\rho},P-\overline{P})\in C([0,T];H^2\cap W^{2,q}),\ (H_t,\rho_t,P_t)\in C([0,T];H^1),\\
& t^{\frac{1}{2}}(H_t,\rho_t,P_t)\in L^\infty([0,T];D^{1,q}),\ u\in C([0,T];H^2)\cap  L^2([0,T];D^3)\cap L^{p_0}([0,T];D^{3,q}),\\
&u_t\in  L^2([0,T];D^1_0)\cap  L^\infty([0,T];L^2),\quad t^{\frac{1}{2}}u\in L^\infty([0,T];D^3),\\
&t^{\frac{1}{2}}u_t\in L^\infty([0,T];D^1_0),\quad t^{\frac{1}{2}}u_{tt}\in L^2([0,T];L^2), \ tu\in L^\infty([0,T]; D^{3,q}),\\
&tu_{t}\in L^\infty([0,T];L^2),\ tu_{tt}\in L^2([0,T];D^1_0)\cap L^\infty([0,T];L^2),\ tu_{ttt}\in L^\infty([0,T];H^{-1}),
\end{split}
\end{equation*}
and  $\rho\geq \underline{\delta}$ on $[0,T]\times \mathbb{R}^3$
for some positive constant $\underline{\delta}$.
\end{lemma}
\begin{proof}
Firstly, we observe the magnetic equations $(\ref{eq:1.2rfv})_1$, it has the form
\begin{equation}\label{ku}
\begin{split}
H_t+\sum_{j=1}^3 A_j \partial_jH+BH=0,
\end{split}
\end{equation}
where $A_j=v_j I_3$ ($j=1,2,3$) are symmetric and $B=\text{div}vI_3-\nabla v$. According to the regularity of $v$ and the standard theory for positive and symmetric hyperbolic system, we easily have the desired conclusions.

Secondly, the existence and regularity of a unique solution $\rho$ to $ (\ref{eq:1.2rfv})_3$ can be obtained essentially according to Lemma 1 in \cite{guahu}. Due to pressure $P$ satisfies the following problem
\begin{equation}\label{pr}
\begin{split}
P_t+v\cdot \nabla P+\gamma P \text{div}v=0, \quad  P_0-\overline{P} \in H^2\cap W^{2,q},
\end{split}
\end{equation}
so we easily have the same conclusions for $P$ via the similar argument as $\rho$.

Finally,  the momentum equations  $ (\ref{eq:1.2rfv})_4$ can be written into
\begin{equation}\label{rst1}
\displaystyle
\rho u_t+Lu=-\nabla P-\rho v\cdot \nabla v+\mu_0\text{rot}H\times H,
\end{equation}
then from Lemma 3 in \cite{guahu},  the desired conclusions is easily obtained.
\end{proof}

\subsection{A priori estimate to the linearized problem away from vacuum}\ \\

Now we want to get some  a priori estimate for the classical solution   $(H,\rho,u)$ to (\ref{eq:1.2rfv})  obtained in Lemma \ref{lem1}, which is independent of the lower bound of the initial mass density $\rho_0$.
For simplicity, we first fix a positive constant $c_0$ sufficiently large that
\begin{equation}\label{houmian}\begin{split}
2+\overline{\rho}+\|(\rho_0-\overline{\rho}, P_0-\overline{P}, H_0)\|_{H^2\cap W^{2,q}}+|u_0|_{D^1_0\cap D^2}+|g_1|_{2}\leq c_0,
\end{split}
\end{equation}
and
\begin{equation}\label{jizhu}\begin{split}
\sup_{0\leq t \leq T^*}|v(t)|^2_{D^1_0\cap D^2}+\int_{0}^{T^*} \Big( |v|^2_{D^3}+|v|^{p_0}_{ D^{3,q}}+|v_t|^2_{D^1_0}\Big)\text{d}t \leq& c_1,\\
\text{ess}\sup_{0\leq t \leq T^*}\Big(t|v_t(t)|^2_{D^1_0}+t|v(t)|^2_{D^3}\Big)+\int_{0}^{T^*} t|v_{t}|^2_{D^2}\text{d}t \leq& c_2,\\
\text{ess}\sup_{0\leq t \leq T^*}\Big(t^2|v(t)|^2_{ D^{3,q}}+t^2|v_t(t)|^2_{D^2}\Big)+\int_{0}^{T^*} t^2|v_{tt}|^2_{D^1_0}\text{d}t \leq& c_3
\end{split}
\end{equation}
for some time $T^*\in (0,T)$ and constants $c_{i}'s$ with $1< c_0\leq c_1 \leq c_2 \leq  c_3 $. Throughout this and next two sections, we denote by $C$ a generic positive constant depending only on fixed constants $\mu$,   $\mu_0$, $T$ and $\lambda$.

Now we give some estimates for the magnetic field $H$.

\begin{lemma}[\textbf{Estimates for magnetic field $H$}]\label{lem:=1}\ \\[2pt]
\begin{equation}
\begin{split}
&\|H(t)\|^2_{H^2\cap W^{2,q}}+ \|H_t(t)\|^2_{1} \leq Cc^4_1, \quad \int_{0}^{t} |H_{tt}|^2_2\text{d}s\leq Cc^3_1,\quad t|H_t(t)|^2_{D^{1,q}}\leq Cc^3_2
\end{split}
\end{equation}
\textrm{for} $0 \leq t \leq T_1=\min(T^*,(1+c_1)^{-1})$.
 \end{lemma}
\begin{proof}
Firstly, let $\alpha=(\alpha_1,\alpha_2,\alpha_3)$ ($|\alpha|\leq 2$) and $\alpha_i=0,1,2$, differentiating $ (\ref{eq:1.2rfv})_1$  $\alpha$ times with respect to $x$, we have
\begin{equation}\label{hyp}\begin{split}
&D^\alpha H_t+\sum_{j=1}^3 A_j \partial_jD^\alpha H+BD^\alpha H \\
=&\big(D^\alpha(BH)-BD^\alpha H\big)+\sum_{j=1}^3 (D^\alpha(A_j \partial_j H)-A_j \partial_jD^\alpha H)=\Theta_1+\Theta_2.
\end{split}
\end{equation}
Then multiplying (\ref{hyp}) by $rD^{\alpha}H|D^{\alpha}H|^{r-2}$ ($r\in [2,q]$) and  integrating  over $\Omega$, we have
\begin{equation}\label{zhenzhen}\begin{split}
\frac{d}{dt}|D^\alpha H|^r_r\leq \Big(\sum_{j=1}^{3}|\partial_{x_j}A_j|_\infty+|B|_\infty\Big)|D^\alpha H|^r_r+|\Theta_1|_r|D^\alpha H|^{r-1}_r+|\Theta_2|_r|D^\alpha H|^{r-1}_r.
\end{split}
\end{equation}

Secondly,  let  $l=r=a$, $b=\infty$ and  $s=|\alpha|=1$  in (\ref{ku22}) of  Lemma \ref{zhen1}, we easily have
\begin{equation}\label{zhen2}\begin{split}
|\Theta_1|_r=|D^\alpha(BH)-BD^\alpha H|_r\leq C|\nabla^2 v|_r|H|_\infty\leq C|\nabla^2 v|_r \|H\|_2;
\end{split}
\end{equation}
 let  $l=r=a$, $b=\infty$ and  $s=|\alpha|=2$  in (\ref{ku22}) of  Lemma \ref{zhen1},  we have
\begin{equation}\label{zhen3}
\begin{split}
|\Theta_1|_r=&|D^\alpha(BH)-BD^\alpha H|_r
\leq C(|\nabla^2 v|_r |\nabla H|_{\infty}+|\nabla^3 v|_r| H|_\infty)\\
\leq&   C\|\nabla^2 v\|_{H^1\cap W^{1,q}} \|H\|_{H^2 \cap W^{2,q} }.
\end{split}
\end{equation}
And similarly, 
 let  $b=\infty$, $l=r=a$ and  $s=|\alpha|=1$  in (\ref{ku22}) of  Lemma \ref{zhen1},  we have
\begin{equation}\label{szhen4}\begin{split}
| D^\alpha(A_j \partial_j H)-A_j \partial_jD^\alpha|_r\leq C|\nabla v|_{r}|\nabla H|_\infty\leq C|\nabla v|_r \| \nabla H\|_{W^{1,q}};
\end{split}
\end{equation}
 let  $a=\infty$, $l=r=b$ and  $s=|\alpha|=2$  in (\ref{ku11}) of  Lemma \ref{zhen1},  we have
\begin{equation}\label{zhen5}\begin{split}
| D^\alpha(A_j \partial_j H)-A_j \partial_jD^\alpha|_r
\leq& C(|\nabla v|_{\infty}|\nabla^2 H|_r+|\nabla^2 v|_r |\nabla H|_{\infty})\\
\leq&  C\|\nabla v\|_2 \|H\|_{ H^2 \cap W^{2,q}}.
\end{split}
\end{equation}
Then combining (\ref{zhenzhen})-(\ref{zhen5}),  according to   Gronwall's inequality, we have
\begin{equation}\label{thzhen1}
\begin{split}
\| H\|_{H^2\cap W^{2,q}} \leq& C\| H_0\|_{H^2\cap W^{2,q}}\text{exp}\Big(\int_0^t \|\nabla v(s)\|_{ H^2\cap W^{2,q}}\text{d}s\Big) \leq Cc_0.
\end{split}
\end{equation}
 for $0\leq t \leq T_1$, where we have used the fact 
\begin{equation}\label{keji}
\begin{split}
&\int_0^t|v(s)|_{D^{3,q}}\text{d}s\leq t^{\frac{1}{q_0}}\Big(\int_0^t |v(s)|^{p_0}_{ D^{3,q}}\text{d}s\Big)^{\frac{1}{p_0}}\leq Cc_1,\ \text{and}\\
&\int_0^t\|\nabla v(s)\|_{2}\text{d}s\leq t^{\frac{1}{2}}\Big(\int_0^t |\nabla v(s)|^2_{2}\text{d}s\Big)^{\frac{1}{2}}\leq C(c_1t+(c_1t)^{\frac{1}{2}})\leq Cc_1,
\end{split}
\end{equation}
and $\frac{1}{p_0}+\frac{1}{q_0}=1$.
Finally, from the magnetic field equations  $ (\ref{eq:1.2rfv})_1$:
$$
H_t=-v\cdot \nabla H-(\text{div}vI_3-\nabla v)H,
$$
we quickly get the desired estimates for $H_t$ and $H_{tt}$.
\end{proof}

Next we give the estimates for the mass density $\rho$ and  pressure $P$.
\begin{lemma}[\textbf{Estimates for the mass density $\rho$ and  pressure $P$}]\label{lem:2}\ \\[4pt]
\begin{equation*}\begin{split}
\|(\rho-\overline{\rho},P-\overline{P})(t)\|_{H^2\cap W^{2,q}}+ \|(\rho_t,P_t)(t)\|_{H^1 \cap L^q}\leq& Cc^2_1,\\
\int_0^t |(\rho_{tt},P_{tt})|^2_2 \text{d}s \leq Cc^3_1, \quad t|(\rho_t,P_t)(t)|^2_{D^{1,q}}\leq& Cc^3_2
\end{split}
\end{equation*}
for  $0\leq t \leq T_1=\min(T^*,C(1+c_1)^{-1})$.
 \end{lemma}
\begin{proof}
From $ (\ref{eq:1.2rfv})_3$ and  the standard energy estimate shown in \cite{CK}, for $2\leq r\leq q$, we have
\begin{equation}\label{kaka4}\begin{split}
\|\rho(t)-\overline{\rho}\|_{W^{2,r}}\leq \Big(\|\rho_0-\overline{\rho}\|_{W^{2,r}}+\overline{\rho}\int_0^t \|\nabla v(s)\|_{W^{2,r}}\text{d}s\Big) \exp\Big(C\int_0^t \|\nabla v(s)\|_{H^2\cap W^{2,q}}\text{d}s\Big).
\end{split}
\end{equation}
Then from  (\ref{keji}), the desired estimate for $\|\rho(t)\|_{H^2 \cap W^{2,q}}$ can be easily obtained via (\ref{kaka4}):
\begin{equation}\label{da90}\begin{split}
\|\rho(t)-\overline{\rho}\|_{H^2\cap W^{2,q}}\leq Cc_0,\ \text{for}\ 0\leq t \leq T_1=\min(T^*,(1+c_1)^{-1}).
\end{split}
\end{equation}

Secondly, the estimates for $(\rho_t,\rho_{tt})$ follows  immediately from the continuity equation
\begin{equation}\label{da99}
\rho_t=-\rho \text{div}v-v\cdot \nabla \rho.
\end{equation}


Finally, due to   pressure $P$ satisfies (\ref{pr}),
then the corresponding estimates for $P$ can be obtained via the same method as $\rho$.

\end{proof}


Now we give the  estimates for the lower order terms of the velocity $u$.

 \begin{lemma}[\textbf{Lower order estimate of the velocity $u$}]\label{lem:4-1}\ \\ 
\begin{equation*}
\begin{split}
|u(t)|^2_{ D^1_0\cap D^2}+|\sqrt{\rho}u_t(t)|^2_{2}+\int_{0}^{t}\big(|u|^2_{D^3}+|u_t|^2_{D^1_0}\big)\text{d}s\leq Cc^{12}_1
\end{split}
\end{equation*}
for $0 \leq t \leq T_2=\min(T^*,C(1+c_1)^{-8})$.
 \end{lemma}
 \begin{proof}
\underline{Step 1}: Multiplying  $ (\ref{eq:1.2rfv})_4$ by $u_t$ and integrating  over $\Omega$, we have
\begin{equation}\label{f1}
\begin{split}
&\int_{\Omega}\rho |u_t|^2 \text{d}x+\frac{1}{2}\frac{d}{dt}\int_{\Omega}\Big(\mu|\nabla u|^2+\big(\lambda+\mu\big)(\text{div}u)^2\Big) \text{d}x\\
=& \int_{\Omega} \Big(-\nabla P-\rho v\cdot \nabla v+(\text{rot}H\times H)\Big)\cdot u_t\text{d}x
=\frac{d}{dt}\Lambda_1(t)- \Lambda_2(t),
\end{split}
\end{equation}
where
\begin{equation*}
\begin{split}
\Lambda_1(t)=&\int_{\Omega} \Big(( P-\overline{P})\text{div}u+(\text{rot}H\times H)\cdot u\Big)\text{d}x,\\
\Lambda_2(t)=&\int_{\Omega} \Big(P_t\text{div}u+\rho (v\cdot \nabla v)\cdot u_t +(\text{rot}H\times H)_t\cdot u\Big)\text{d}x.
\end{split}
\end{equation*}
According to Lemmas \ref{lem:=1}-\ref{lem:2}, Holder's inequality, Gagliardo-Nirenberg inequality and Young's inequality, we easily deduce that
\begin{equation*}
\begin{split}
\Lambda_1(t)\leq& C\big(|\nabla u|_2 |P-\overline{P}|_2+|\nabla H|_2 |H|_{3} |\nabla u|_{2}\big)
\leq  \frac{\mu}{10}|\nabla u|^2_2+Cc^{8}_1,\\
\Lambda_2(t)\leq&  C\big(|\nabla u|_2 |P_t|_2+|\rho|^{\frac{1}{2}}_{\infty} |\sqrt{\rho}u_t|_2  |v|_{\infty}  |\nabla v|_2+\|H\|_{2} \|H_t\|_1 |\nabla u|_{2}\big)\\
\leq&  C|\nabla u|^2_2+\frac{1}{10}|\sqrt{\rho} u_t|^2+Cc^{8}_1
\end{split}
\end{equation*}
for $0< t\leq T_1$.
Then integrating (\ref{f1}) over $(0,t)$ with respect to $t$, we have
\begin{equation*}
\begin{split}
\int_{0}^{t}|\sqrt{\rho} u_t(s)|^2_2\text{d}s+|\nabla u(t)|^2_2\leq C\int_{0}^{t}|\nabla u(s)|^2_2\text{d}s+Cc^{8}_1
\end{split}
\end{equation*}
for $0 \leq t \leq T_1$,
via Gronwall's inequality, we have
\begin{equation}\label{plh}
\begin{split}
&\int_{0}^{t}|\sqrt{\rho} u_t(s)|^2_2\text{d}s+|\nabla u(t)|^2_2
\leq Cc^{8}_1\exp\big(Ct\big)\leq Cc^{8}_1, \quad 0\leq t \leq T_1.
\end{split}
\end{equation}

Combining Lemmas \ref{lem:=1}-\ref{lem:2} and Lemma \ref{ok}, we easily have
\begin{equation}\label{plmn}
\begin{split}
\int_{0}^{t} |u|^2_{D^2}\text{d}s\leq&C \int_{0}^{t}\Big(|\rho u_t+\rho v\cdot \nabla v|^2_2+|\nabla P|^2_2+|\text{rot}H\times H|^2_2+|u|^2_{D^1_0} \Big)\text{d}s\leq Cc^{10}_1.
\end{split}
\end{equation}

%
%
%
%
%
%
%
%
%
%
\underline{Step 2}: Differentiating $ (\ref{eq:1.2rfv})_4$ with respect to $t$, we have
\begin{equation}\label{gh78}
\begin{split}
\rho u_{tt}+Lu_t= -\nabla P_t-\rho_tu_t-(\rho v\cdot\nabla v)_t
 +(\text{rot}H\times H)_t.
\end{split}
\end{equation}
Multiplying (\ref{gh78}) by $u_t$ and integrating (\ref{gh78}) over $\Omega$, we have
\begin{equation}\label{zhen4}
\begin{split}
&\frac{1}{2}\frac{d}{dt}\int_{\Omega}\rho |u_t|^2 \text{d}x+\int_{\mathbb{R}^3}(\mu|\nabla u_t|^2+(\lambda+\mu)(\text{div}u_t)^2) \text{d}x\\
=&  \int_{\Omega} \big(-\nabla P_t-( \rho  v \cdot \nabla v)_t-\frac{1}{2}\rho_t u_t+ (\text{rot}H\times H)_t\big)\cdot u_t\text{d}x
\equiv:\sum_{i=1}^{4}I_i.
\end{split}
\end{equation}

According to Lemmas \ref{lem:=1}-\ref{lem:2}, Holder's inequality, Gagliardo-Nirenberg inequality and Young's inequality,  we deduce that
\begin{equation}\label{zhou6}
\begin{split}
I_1=&\int_{\Omega} P_t\text{div}u_t \text{d}x\leq C|P_t|_{2} |\nabla u_t|_{2}\leq \frac{\mu}{10}|\nabla u_t|^2_2+Cc^4_1,\\
I_2\leq& C|\rho|^{\frac{1}{2}}_{\infty} |\nabla v_t|_2 |\nabla v|_{3} |\sqrt{\rho}u_t|_{2}+|\rho|^{\frac{1}{2}}_{\infty} |v|_{\infty} |\nabla v_t|_{2} |\sqrt{\rho}u_t|_{2}+C|\rho_t|_{3} |v|_{\infty} |\nabla v|_{2} |  u_t|_{6}\\
\leq &C |\sqrt{\rho}u_t|^2_{2}+\frac{\mu}{10}|\nabla u_t|^2_2+Cc^4_1(1+|\nabla v_t|^2_2),\\
I_3=&-\frac{1}{2}\int_{\Omega} \rho_t |u_t|^2  \text{d}x=\int_{\Omega} \rho v  u_t \cdot \nabla u_t  \text{d}x\leq  C|\rho|^{\frac{1}{2}}_{\infty}|v|_{D^1_0}|\sqrt{\rho} u_t|_{3}|\nabla u_t|_{2}\\
\leq& Cc^8_1|\sqrt{\rho} u_t|^2_{2}+\frac{\mu}{10}|\nabla u_t|^2_2,\\
I_4=&
\int_{\Omega} \text{div}\Big(H\otimes H-\frac{1}{2}|H|^2I_3\Big)_t\cdot u_t\text{d}x
=-\int_{\Omega} (H\otimes H-\frac{1}{2}|H|^2I_3)_t:\nabla u_t\text{d}x \\
\leq& C|\nabla u_t|_2 |H_t|_2 |H|_{\infty}\leq Cc^8_1+\frac{\mu}{10}|\nabla u_t|^2_2.
\end{split}
\end{equation}
Then combining the above estimate (\ref{zhou6})  and  (\ref{zhen4}), we have
\begin{equation}\label{zhen5g}
\begin{split}
&\frac{1}{2}\frac{d}{dt}\int_{\Omega}\rho |u_t|^2 \text{d}x+\int_{\Omega}|\nabla u_t|^2\text{d}x
\leq Cc^{8}_1|\sqrt{\rho}u_t|^2_{2}+Cc^4_1|\nabla v_t|^2_{2}+Cc^{8}_1.
\end{split}
\end{equation}
Integrating (\ref{zhen5g}) over $(\tau,t)$ ($\tau\in (0,t)$), for $\tau\leq t \leq T_1$, we  have
\begin{equation}\label{nvk3}
\begin{split}
&|\sqrt{\rho}u_t(t)|^2_2 +\int_{\tau}^{t}|\nabla u_t|^2_{2}\text{d}s
\leq |\sqrt{\rho}u_t(\tau)|^2_2+Cc^{8}_1\int_{\tau}^{t} |\sqrt{\rho}u_t|^2_2 \text{d}s+Cc^{8}_1.
\end{split}
\end{equation}
From the momentum equations  $(\ref{eq:1.2rfv})_4$, we easily have
\begin{equation}\label{li9}
\begin{split}
|\sqrt{\rho}u_t(\tau)|^2_2\leq C\int_{\Omega} \rho |v|^2|\nabla v|^2\text{d}x+C\int_{\Omega} \frac{|\nabla P+Lu- \text{rot}H\times H|^2}{\rho}\text{d}x,
\end{split}
\end{equation}
due to the initial layer compatibility condition (\ref{th79}), letting $\tau\rightarrow 0$ in (\ref{li9}), we have
\begin{equation}\label{nvk33}
\begin{split}
\lim \sup_{\tau\rightarrow 0}|\sqrt{\rho}u_t(\tau)|^2_2 \leq C\int_{\Omega} \rho_0 |u_0|^2|\nabla u_0|^2\text{d}x+C\int_{\Omega} |g_1|^2\text{d}x\leq Cc^4_0.
\end{split}
\end{equation}
Then, letting $\tau\rightarrow 0$ in (\ref{nvk3}), we have
\begin{equation}\label{nvk44}
\begin{split}
&|\sqrt{\rho}u_t(t)|^2_2 +\int_{0}^{t}|\nabla u_t|^2_{2}\text{d}s \leq Cc^{8}_1+Cc^{8}_1\int_{0}^{t} |\sqrt{\rho}u_t|^2_2 \text{d}s.
\end{split}
\end{equation}
From Gronwall's inequality, we deduce that
\begin{equation}\label{nvk55}
\begin{split}
|\sqrt{\rho}u_t(t)|^2_2 +\int_{0}^{t}|\nabla u_t|^2_{2}\text{d}s\leq Cc^{8}_1\exp(Cc^{8}_1t)\leq Cc^{8}_1,\ 0 \leq t \leq T_2.
\end{split}
\end{equation}

Finally, due to 
Lemmas \ref{lem:=1}-\ref{lem:2} and  Lemma \ref{ok}, for $0 \leq t \leq T_2$, we easily have
\begin{equation*}
\begin{split}
&\ \quad| u(t)|_{D^2}\leq \big(|\rho u_t(t)+\rho v\cdot \nabla v(t)|_2 +|\nabla P(t)|_2+|\text{rot}H\times H(t)|_2+|u(t)|_{D^1_0}\big)\leq Cc^{5}_1,\\
&\int_{0}^{t}| u|^2_{D^3}\text{d}s\leq C\int_{0}^{t}\Big(|\rho u_t+\rho v\cdot \nabla v|^2_{D^1} +|\nabla P|^2_{D^1}+|\text{rot}H\times H|^2_{D^1}+|u|^2_{D^1_0}\Big)\text{d}s
\leq Cc^{12}_1.
\end{split}
\end{equation*}
\end{proof}
Now we will give some estimates for the higher order terms of the velocity $u$ in the following three Lemmas.
\begin{lemma}[\textbf{Higher order estimate of the velocity $u$}]\label{lem:4-3}\ \\
\begin{equation*}
\begin{split}
t|u_t(t)|^2_{D^1_0}+t|u(t)|^2_{D^3}+\int_{0}^{t}s\big(|u_t|^2_{D^{2}}+|\sqrt{\rho} u_{tt}|^2_{2}\big)\text{d}s\leq Cc^{24}_2, \quad 0\leq t \leq T_2.
\end{split}
\end{equation*}
 \end{lemma}
\begin{proof}
Multiplying  (\ref{gh78}) by $u_{tt}$ and integrating  over $\Omega$, we have
\begin{equation}\label{f3}
\begin{split}
&\int_{\Omega}\rho |u_{tt}|^2 \text{d}x+\frac{1}{2}\frac{d}{dt}\int_{\Omega}\Big(\mu|\nabla u_t|^2+(\lambda+\mu)(\text{div}u_t)^2\Big) \text{d}x\\
=& \int_{\Omega} \big(-\nabla P_t-(\rho v\cdot \nabla v)_t-\rho_tu_t+(\text{rot}H\times H)_t\big)\cdot u_{tt}\text{d}x
=\frac{d}{dt}\Lambda_3(t)+ \Lambda_4(t),
\end{split}
\end{equation}
where
\begin{equation*}
\begin{split}
\Lambda_3(t)=&\int_{\Omega} \big( P_t\text{div}u_t-\rho_t(v\cdot \nabla v)\cdot u_t-\frac{1}{2}\rho_t|u_t|^2+(\text{rot}H\times H)_t \cdot u_t \big)\text{d}x,\\
\Lambda_4(t)=&\int_{\Omega} \big(-P_{tt}\text{div}u_t-\rho (v\cdot \nabla v)_t \cdot u_{tt}+\rho_{tt} (v\cdot \nabla v) \cdot u_t
 +\rho_{t} (v\cdot \nabla v)_t \cdot u_t\big)\text{d}x\\
 &+\int_{\Omega}\big( \frac{1}{2}\rho_{tt}|u_t|^2- (\text{rot}H\times H)_{tt}\cdot u_t\big)\text{d}x\equiv:\sum_{i=5}^{10}I_i.
\end{split}
\end{equation*}
Then  almost same to (\ref{zhou6}), we also have
\begin{equation}\label{zhen6}\begin{split}
\Lambda_3(t)\leq \frac{\mu}{10} |\nabla u_t|^2_2+Cc^8_1|\sqrt{\rho}u_t|^2_2+Cc^8_1\leq \frac{\mu}{10} |\nabla u_t|^2_2+Cc^{20}_1, \quad 0\leq t \leq T_2.
\end{split}
\end{equation}
Let we denote
$$
\Lambda^*(t)=\frac{1}{2}\int_{\Omega}\mu|\nabla u_t|^2+(\lambda+\mu)(\text{div}u_t)^2\text{d}x-\Lambda_3(t),
$$
then from  (\ref{zhen6}), for $0 \leq t \leq T_2$, we quickly  have
\begin{equation}\label{kui}
\begin{split}
C|\nabla u_t|^2_{2}-Cc^{20}_1 \leq \Lambda^*(t)\leq&  C|\nabla u_t|^2_{2}+Cc^{20}_1.
\end{split}
\end{equation}

Similarly,  from Holder's inequality and  Gagliardo-Nirenberg inequality, for $0 < t \leq T_2$,  we deduce that
\begin{equation}\label{ght}
\begin{split}
&I_{5}\leq C|P_{tt}|_{2}  | \nabla u_t|_{2},\ I_{6} \leq |\rho|^{\frac{1}{2}}_{\infty} |\sqrt{\rho}u_{tt}|_2\big(|v|_{\infty} |\nabla v_t|_2+|\nabla v|_3 |\nabla v_t|_2\big),\\
&I_{7}\leq C|\rho_{tt}|_2 |\nabla u_t|_2 |\nabla v|_3 |v|_{\infty}, \\
& I_{8}\leq C|\rho_t|_2 |v_{t}|_6 |\nabla v|_{6} |\nabla u_t|_2+C|v|_{\infty}|v_t|_{6}|\nabla u_t|_{2}|\rho_t|_3,\\
&I_{9}\leq C|\rho_{t}|_3 |\nabla u_t|_2|v|_{\infty}|u_t|_{6}+C|\rho|^{\frac{1}{2}}_{\infty} |\sqrt{\rho}u_t|_{3} | v_t|_{6} |\nabla u_t|_2,
\end{split}
\end{equation}
where we have used the facts $\rho_t=-\text{div}(\rho v)$,
and
\begin{equation}\label{zhen7}
\begin{split}
 I_{10}=&-\int_{\Omega}(\text{rot}H\times H)_{tt}\cdot u_{t}\text{d}x
=\int_{\Omega}\big(H\otimes H-\frac{1}{2}|H|^2I_3\big)_{tt}: \nabla u_{t}\text{d}x\\
\leq& C|\nabla u_t|_2|H_t|^2_4+C|\nabla u_t|_2|H_{tt}|_2| H|_{\infty}.
\end{split}
\end{equation}
Combining (\ref{ght})-(\ref{zhen7})  and  Lemmas \ref{lem:=1}-\ref{lem:4-1}, from Young's inequality, we have
\begin{equation}\label{zhen10}\begin{split}
\Lambda_4(t) \leq&\frac{1}{2}|\sqrt{\rho} u_{tt}(t)|^2_{2}+Cc^{8}_1(1+|v_t|^2_{D^1_0})|\nabla u_t|^2_{2}\\
&+Cc^4_1(1+|P_{tt}|^2_2+|\rho_{tt}|^2_2+|H_{tt}|^2_2)+Cc^{18}_1|v_t|^2_{D^1_0}.
\end{split}
\end{equation}
Then multiplying (\ref{f3}) with $t$ and  integrating   over $(\tau,t)$ ($\tau\in (0,t)$), from (\ref{kui}) and (\ref{zhen10}),
we have
\begin{equation}
\label{e2}
\begin{split}
&\int_{\tau}^{t}s|\sqrt{\rho}u_{tt}(s)|^2_2\text{d}s+t|\nabla u_t(t)|^2_{2}\\
\leq& \tau| u_t(\tau)|^2_{D^1_0}
+Cc^8_1\int_{\tau}^{t}s(1+|\nabla v_t|^2_2)|\nabla u_t|^2_{2}\text{d}s+Cc^{20}_2
\end{split}
\end{equation}
for $\tau\leq t \leq T_2$.  From Lemma \ref{lem:4-1}, we have $\nabla u_{t}\in L^2([0,T_2];L^2)$, then  according to Lemma \ref{bei},  there exists a sequence ${s_k}$ such that
$$
s_k\rightarrow 0,\quad \text{and} \quad s_k|\nabla u_{t}(s_k)|^2_2\rightarrow 0, \quad \text{as} \quad k\rightarrow \infty.
$$
Therefore, letting $\tau=s_k\rightarrow 0$ in (\ref{e2}), we conclude that
\begin{equation}
\label{ew2}
\begin{split}
&\int_{0}^{t}s|\sqrt{\rho}u_{tt}(s)|^2_2\text{d}s+t|\nabla u_t(t)|^2_{2}
\leq Cc^{8}_1\int_{0}^{t}s(1+|\nabla v_t|^2_2)|\nabla u_t|^2_{2}\text{d}s+Cc^{20}_2.
\end{split}
\end{equation}
Then from Gronwall's inequality, we have
\begin{equation*}
\begin{split}
&\int_{0}^{t}s|\sqrt{\rho}u_{tt}(s)|^2_2\text{d}s+t| u_t(t)|^2_{D^1_0}
\leq Cc^{20}_2\exp\Big(Cc^{8}_1\int_{0}^{t}s(1+|\nabla v_t|^2_2)\text{d}s\Big)\leq Cc^{20}_2.
\end{split}
\end{equation*}
Finally, from Lemma \ref{ok},
 for $0 \leq t \leq T_2$,  we immediately have
\begin{equation*}\begin{split}
t|u(t)|^2_{D^3} \leq t\big(|\rho u_t+\rho v\cdot \nabla v|^2_{D^1} +|\nabla P|^2_{D^1}+|\text{rot}H\times H|^2_{D^1}+|u|^2_{D^1_0}\big) \leq Cc^{24}_2,
\end{split}
\end{equation*}
and similarly,  
\begin{equation*}\begin{split}
&\int_{0}^{t}s|u_{t}|^2_{D^2}\text{d}s
\leq C\int_{0}^{t} s\big(|(\rho u_t+\rho v\cdot \nabla v)_t|^2_{2} +|\nabla P_t|^2_{2}+|(\text{rot}H\times H)_t|^2_{2}+|u_t|^2_{D^1_0}\big)\text{d}s \leq Cc^{22}_2.
\end{split}
\end{equation*}
\end{proof}

\begin{lemma}[\textbf{Higher order estimate of the velocity $u$}]\label{lem:4-4*}
\begin{equation*}
\begin{split}
\int_{0}^{t}|u(s)|^{p_0}_{D^{3,q}}\text{d}s\leq Cc^{54}_2\quad  \text{for}\quad  0\leq t \leq T_2.
\end{split}
\end{equation*}
 \end{lemma}

\begin{proof}
From $ (\ref{eq:1.2rfv})_4$, via Lemma \ref{ok}, Holder's inequality and  Gagliardo-Nirenberg inequality,  we easily deduce that
\begin{equation}\label{zhen11}\begin{split}
|u|_{D^{3,q}}
\leq& \big(|\rho u_t+\rho v\cdot \nabla v|_{D^{1,q}} +|\nabla P|_{D^{1,q}}+|\text{rot}H\times H|_{D^{1,q}}+|u|_{D^{1,q}_0}\big)\\
\leq &   C(c^6_1+c^2_1|u_t|_\infty+c^2_1|\nabla u_t|_q+c^3_1|v|_{D^{2,q}}).
\end{split}
\end{equation}
Due to the Sobolev inequality and Young's inequality, we have
\begin{equation*}\begin{cases}
\displaystyle
|u_t|_\infty\leq C|u_t|^{1-\frac{3}{q}}_{q}\| u_t\|^{\frac{3}{q}}_{W^{1,q}}\leq C|\nabla u_t|_2 +C|\nabla u_t|_q, \ \text{when}\  \Omega \ is \  bounded,\\[8pt]
\displaystyle
|u_t|_\infty\leq C|u_t|^{\frac{6(q-3)}{3q+6(q-3)}}_{6}|\nabla u_t|^{\frac{3q}{3q+6(q-3)}}_{q}\leq C|\nabla u_t|_2 +C|\nabla u_t|_q, \ \text{when}\  \Omega=\mathbb{R}^3.
\end{cases}
\end{equation*}
Then we quickly obtain
\begin{equation*}\begin{split}
|u(t)|_{D^{3,q}}
\leq & Cc^2_1(|\nabla u_t|_2+|\nabla u_t|_q)+Cc^3_1|v|_{D^{2,q}}+Cc^{6}_1.
\end{split}
\end{equation*}
According to Lemmas \ref{lem:=1}-\ref{lem:4-3}, we have
\begin{equation}\label{moufan}\begin{split}
\int_{0}^{t}|u|^{p_0}_{D^{3,q}}\text{d}s
\leq & Cc^{12}_1+Cc^6_1\int_{0}^{t}\big(  | v|^{p_0}_{D^{2,q}}+|\nabla u_t|^{p_0}_2+|\nabla u_t|^{p_0}_q\big)\text{d}s\\
\leq& Cc^{12}_1+Cc^6_1\int_{0}^{t} |\nabla u_t|^{\frac{p_0(6-q)}{2q}}_{2}|\nabla u_t|^{\frac{p_0(3q-6)}{2q}}_{6}\text{d}s\\
\leq& Cc^{12}_1+Cc^6_1\int_{0}^{t}s^{-\frac{p_0}{2}} \big(s|\nabla u_t|^2_{2}\big)^{\frac{p_0(6-q)}{4q}}\big(s| u_t|^2_{D^2}\big)^{\frac{p_0(3q-6)}{4q}}\text{d}s\\
\leq& Cc^{12}_1+Cc^6_1\big(\sup_{[0,T_2]}s|\nabla u_t|^2_{2}\big)^{\frac{p_0(6-q)}{4q}}\int_{0}^{t}s^{-\frac{p_0}{2}} \big(s| u_t|^2_{D^2}\big)^{\frac{p_0(3q-6)}{4q}}\text{d}s\\
\leq& Cc^{12}_1+Cc^{30}_2\Big(\int_{0}^{t}s^{-\frac{2p_0q}{4q-p_0(3q-6)}} \text{d}s\Big)^{\frac{4q-p_0(3q-6)}{4q}}\Big(\int_{0}^{t} s| u_t|^2_{D^2}\text{d}s\Big)^{\frac{p_0(3q-6)}{4q}}\\
\leq& Cc^{54}_2
\end{split}
\end{equation}
due to $0<\frac{2p_0q}{4q-p_0(3q-6)}<1$ and $0<\frac{p_0(3q-6)}{4q}<1$.
\end{proof}
\begin{lemma}[\textbf{Higher order estimate of the velocity $u$}]\label{lem:4-4}
\begin{equation*}
\begin{split}
t^2|u(t)|_{D^{3,q}}+t^2|u_t(t)|^2_{ D^2}+t^2|\sqrt{\rho}u_{tt}(t)|^2_2+\int_{0}^{t}s^2|u_{tt}(s)|^2_{D^1_0}\text{d}s\leq Cc^{34}_3
\end{split}
\end{equation*}
for $0\leq t \leq T_3=\min(T^*,(1+c_3)^{-8})$.
 \end{lemma}
 \begin{proof}
Differentiating the equations  (\ref{gh78}) with respect to $t$, we have
%
%
%
\begin{equation}\label{e3}
\begin{split}
\rho u_{ttt}+Lu_{tt}=&-\nabla P_{tt}-\rho (v\cdot \nabla v)_{tt}-2\rho_t(v\cdot \nabla v+u_t)_t\\
&-\rho_{tt}(v\cdot \nabla v+u_t)+(\text{rot}H\times H)_{tt}.
\end{split}
\end{equation}
Multiplying (\ref{e3}) by $u_{tt}$ and integrating  over $\Omega$, we have
\begin{equation}\label{ebn}
\begin{split}
&\frac{1}{2}\frac{d}{dt}\int_{\Omega}\rho |u_{tt}|^2 \text{d}x+\int_{\Omega}(\mu|\nabla u_{tt}|^2+(\lambda+\mu)(\text{div}u_{tt})^2) \text{d}x\\
=&  \int_{\Omega}\Big( P_{tt}\text{div}u_{tt}-\rho(v\cdot \nabla v)_{tt}\cdot u_{tt}-2\rho_t(v\cdot \nabla v)_t\cdot u_{tt}-\rho_{tt}(v\cdot \nabla v )\cdot u_{tt}\Big)\text{d}x\\
&+\int_{\Omega}\Big(-\frac{3}{2}\rho_t|u_{tt}|^2-\rho_{tt}u_t\cdot u_{tt}+(\text{rot}H\times H)_{tt}\cdot u_{tt}\Big)\text{d}x
=\Lambda_5(t)\equiv :\sum_{i=11}^{17}I_i.
\end{split}
\end{equation}
From Lemmas \ref{lem:=1}-\ref{lem:4-4*}, Holder's inequality and  Gagliardo-Nirenberg inequality,  we obtain
\begin{equation}\label{e4}
\begin{split}
I_{11}\leq&C |P_{tt}|_{2} |\nabla u_{tt}|_{2},\ I_{12}\leq C|\rho|^{\frac{1}{2}}_{\infty}|\sqrt{\rho}u_{tt}|_{2}  \big( |v_{tt}|_{D^1_0}\|\nabla v\|_{1}+| v_t|_{D^1_0}\|\nabla v_t\|_{1}\big),\\
I_{13}\leq&C |\rho_t|_{3} \|\nabla v\|_{1} | v_t|_{D^1_0} |\nabla u_{tt}|_2,\ I_{14}\leq C|\rho_{tt}|_2 \|\nabla v\|^2_{1} |\nabla u_{tt}|_2,\\
I_{15}\leq&C|\rho|^{\frac{1}{2}}_{\infty}\|\nabla v\|_{1}|\sqrt{\rho} u_{tt}|_2|\nabla u_{tt}|_2,\
I_{16}\leq C|\rho_t|_{3} \|\nabla v\|_{1} |u_t|_{D^1_0}|\nabla u_{tt}|_2\\
&+C|\rho|^{\frac{3}{4}}_{\infty} |v_t|_{D^1_0}  \big(|u_t|^{\frac{1}{2}}_{D^1_0} |\sqrt{\rho}u_t|^{\frac{1}{2}}_2|\nabla u_{tt}|_2+|u_t|_{D^1_0} |\sqrt{\rho}u_{tt}|^{\frac{1}{2}}_2|\nabla u_{tt}|^{\frac{1}{2}}_2\big),
\end{split}
\end{equation}
where we have used the fact that $\rho_t=\text{div}(\rho v)$, and
\begin{equation}\label{eh4}
\begin{split}
I_{17}=&-\int_{\Omega}(\text{rot}H\times H)_{tt}\cdot u_{tt}\text{d}x
=\int_{\Omega}\big(H\otimes H-\frac{1}{2}|H|^2I_3\big)_{tt}: \nabla u_{tt}\text{d}x\\
\leq& C|\nabla u_{tt}|_2|H_t|^2_4+C|\nabla u_{tt}|_2|H_{tt}|_2| H|_{\infty}.
\end{split}
\end{equation}
Then from Young's inequality, the above estimates (\ref{e4})-(\ref{eh4}) imply that
\begin{equation}\label{fan}
\begin{split}
t^2\Lambda_5(t) \leq& \frac{\mu}{2}t^2|u_{tt}|^2_{D^1_0}+C(c^6_3+c^2_3|v_t|^2_{D^1_0}) t^2|\sqrt{\rho}u_{tt}|^2_2+Cc^4_3t^2(|\rho_{tt}|^2_{2}+|P_{tt}|^2_{2})\\
&+Cc^4_3t^2|H_{tt}|^2_{2}+Ct^2(|v_{tt}|^2_{D^1_0}+|v_t|^2_{D^2})+Cc^{6}_3t^2|u_t|^2_{D^2}+Cc^{30}_3.
\end{split}
\end{equation}
Then multiplying (\ref{ebn}) by $t^2 $ and integrating  over $(\tau,t)$ ($\tau\in (0,t)$), we obtain
\begin{equation}\label{nvk3355}
\begin{split}
&t^2|\sqrt{\rho}u_{tt}(t)|^2_2 +\int_{\tau}^{t}s^2|\nabla u_{tt}|^2_2 \text{d}s\\
\leq& \tau^2|\sqrt{\rho}u_{tt}(\tau)|^2_2+C\int_{\tau}^{t}(c^6_3+c^2_3|v_t|^2_{D^1_0}) s^2|\sqrt{\rho}u_{tt}|^2_2\text{d}s +Cc^{30}_3\\
\end{split}
\end{equation}
for $\tau \leq t\leq T_2$. Due to Lemma \ref{lem:4-3}, we have $t^{\frac{1}{2}}\sqrt{\rho}u_{tt}\in L^2([0,T_2];L^2)$, then from Lemma \ref{bei}, there exists a sequence ${s_k}$ such that
$$
s_k\rightarrow 0,\quad \text{and} \quad s^2_k|\sqrt{\rho}u_{tt}(s_k)|^2_2\rightarrow 0, \quad \text{as} \quad k\rightarrow \infty.
$$
Therefore, letting $\tau=s_k\rightarrow 0$ in (\ref{nvk3355}), we conclude that
\begin{equation}\label{nvk3344}
\begin{split}
&t^2|\sqrt{\rho}u_{tt}(t)|^2_2 +\int_{0}^{t}s^2|\nabla u_{tt}|^2_2 \text{d}s
\leq  C\int_{\tau}^{t}(c^6_3+c^2_3|v_t|^2_{D^1_0}) s^2|\sqrt{\rho}u_{tt}|^2_2\text{d}s +Cc^{30}_3.
\end{split}
\end{equation}
Via the Gronwall's inequality,  for $0\leq t \leq T_3$, we have
\begin{equation*}
\begin{split}
&t^2|\sqrt{\rho}u_{tt}(t)|^2_2 +\int_{0}^{t}s^2|\nabla u_{tt}|^2_2 \text{d}s
\leq Cc^{30}_3 \text{exp}\Big(\int_{\tau}^{t}(c^6_3+c^2_3|v_t|^2_{D^1_0}) \text{d}s\Big)\leq Cc^{30}_3.
\end{split}
\end{equation*}
Moreover, from Lemma \ref{ok}  and  (\ref{zhen11}), we quickly have
\begin{equation}\label{wwzhen11}\begin{split}
t^2|u_{t}|^2_{D^2}
\leq& Ct^2\big(|(\rho u_t+\rho v\cdot \nabla v)_t|^2_{2} +|\nabla P_t|^2_{2}+|(\text{rot}H\times H)_t|^2_{2}+|u_t|^2_{D^1_0}\big) \leq Cc^{32}_3,\\
t^2|u|_{D^{3,q}}
\leq &   Ct^2(c^4_1+c^2_1|u_t|_\infty+c^2_1|\nabla u_t|_q+Cc^3_1|v|_{D^{2,q}})\leq Cc^{34}_3.
\end{split}
\end{equation}
\end{proof}
Then combining the above lemmas,  for $0\leq t\leq T_*=\min(T^*,(1+c_3)^{-8})$,  we have the following a priori estimate:
\begin{equation}\label{hym}
\begin{split}
\|(H,\rho-\overline{\rho},P-\overline{P})(t)\|_{H^2\cap W^{2,q}}+ \|(H_t,\rho_t,P_t)(t)\|_{H^1\cap L^q}\leq& Cc^2_1,\\
\int_0^t |(H_{tt},\rho_{tt},P_{tt})|^2_2 \text{d}s+t|(H_t, \rho_t,P_t)(t)|^2_{D^{1,q}}\leq& Cc^3_2,\\
|u(t)|^2_{D^1_0\cap D^2}+|\sqrt{\rho}u_t(t)|^2_{2}+\int_{0}^{t}\Big(|u_t|^2_{D^1_0}+|u|^2_{D^3}\Big)\text{d}s\leq& Cc^{12}_1,\\
t|u_t(t)|^2_{D^1_0}+t|u(t)|^2_{D^3}+\int_{0}^{t}\Big(|u|^{p_0}_{D^{3,q}}+s\big(|u_t|^2_{D^{2}}+|\sqrt{\rho} u_{tt}|^2_{2}\big)\Big)\text{d}s\leq& Cc^{54}_2,\\
t^2|u(t)|^2_{D^{3,q}}+t^2|u_t(t)|^2_{ D^2}+t^2|\sqrt{\rho}u_{tt}(t)|^2_2+\int_{0}^{t}s^2|u_{tt}|^2_{D^1_0}\text{d}s\leq& Cc^{34}_3.
\end{split}
\end{equation}

\subsection{Unique solvability of the IBVP (\ref{eq:1.2rfv}) and (\ref{th78})-(\ref{th79}) with vacuum}\ \\

In this section, we will construct a sequence of approximation solutions to the linearized problem  (\ref{eq:1.2rfv}) with vacuum.
\begin{lemma}\label{lembbn}Let (\ref{ghk1}) and (\ref{houmian})-(\ref{jizhu}) hold.
 Assume $(H_0,\rho_0,u_0)$ satisfies (\ref{th78})-(\ref{th79}).
Then there exists a unique classical solution $(H,\rho,u)$ to  (\ref{eq:1.2rfv}) satisfying
\begin{equation}\label{gujq}\begin{split}
&(H,\rho-\overline{\rho},P-\overline{P})\in C([0,T_*];H^2\cap W^{2,q}),\\
&u\in C([0,T_*];D^1_0\cap D^2)\cap  L^2([0,T_*];D^3)\cap L^{p_0}([0,T_*];D^{3,q}),\ u_t\in  L^2([0,T_*];D^1_0),\\
&\sqrt{\rho}u_t\in L^\infty([0,T_*];L^2),\ t^{\frac{1}{2}}u\in L^\infty([0,T_*];D^3),\ t^{\frac{1}{2}}\sqrt{\rho}u_{tt}\in L^2([0,T_*];L^2),\\
&t^{\frac{1}{2}}u_t\in L^\infty([0,T_*];D^1_0)\cap L^2([0,T_*];D^2),\ tu\in L^\infty([0,T_*]; D^{3,q}),\\
& tu_{tt}\in L^2([0,T_*];D^1_0),\ tu_t\in L^\infty([0,T_*]; D^{2}),\ t\sqrt{\rho}u_{tt}\in L^\infty([0,T_*];L^2).
\end{split}
\end{equation}
Moreover, the solution $(H,\rho,u)$ also satisfies the estimate (\ref{hym}).
\end{lemma}
\begin{proof}
\underline{Step 1}: Existence.
We define $\rho_0=\rho_0+\delta$ for each $\delta\in (0,1)$. Then from the compatibility condition (\ref{th79}), we have
\begin{equation*}
\begin{split}
&Lu_0+\nabla P(\rho^\delta_0)-\mu_0\text{rot}H_0\times H_0=(\rho^\delta_0)^{\frac{1}{2}}_0 g^\delta_1,
\end{split}
\end{equation*}
where
\begin{equation*}\begin{split}
g^\delta_1=&\Big(\frac{\rho_0}{\rho^\delta_0}\Big)^{\frac{1}{2}}g_1+\frac{\nabla (P(\rho^\delta_0)-P(\rho_0))}{(\rho^\delta_0)^{\frac{1}{2}}}.
\end{split}
\end{equation*}

Then according to  assumption (\ref{houmian}), for  sufficiently small $\delta> 0$, we have
\begin{equation*}\begin{split}
1+\overline{\rho}+\delta+\|(\rho^\delta_0-(\overline{\rho}+\delta),P(\rho^\delta_0)-P(\overline{\rho}+\delta),H_0)\|_{H^2\cap W^{2,q}}+|u_0|_{D^1_0\cap D^2}+|g^\delta_1|_{2}\leq c_0.
\end{split}
\end{equation*}
Therefore, corresponding to $(H_0,\rho^\delta_0,P(\rho^\delta_0),u_0)$, there exists a unique classical solution $(H^\delta,\rho^\delta,P^\delta,u^\delta)$ satisfying (\ref{hym}). Then there exists a subsequence of solutions $(H^\delta,\rho^\delta,P^\delta,u^\delta)$ converges to a limit $(H,\rho,P,u)$ in weak or weak* sense. And for any $R> 0$, due to Lemma \ref{aubin}, there exists a subsequence of solutions $(H^\delta,\rho^\delta,P^\delta,u^\delta)$ satisfying
\begin{equation}\label{ert}\begin{split}
&(H^\delta,\rho^\delta,P^\delta,u^\delta)\rightarrow (H,\rho,P,u)\ \text{in } C([0,T_*];H^1(\Omega_R)),
\end{split}
\end{equation}
where $\Omega_R=\Omega \cap B_R$. Combining the lower semi-continuity of norms and (\ref{ert}), we know that $(H,\rho,P,u)$ also satisfies the local estimates (\ref{hym}).
So it is easy to show that $(H,\rho,P,u)$ is a solution in distribution sense and  satisfies the regularity
\begin{equation}\label{yong}\begin{split}
&(H,\rho-\overline{\rho},P-\overline{P})\in L^\infty([0,T_*];H^2\cap W^{2,q}),\\
&u\in L^\infty([0,T_*];D^1_0\cap D^2)\cap  L^2([0,T_*];D^3)\cap L^{p_0}([0,T_*];D^{3,q}),\\
& u_t\in  L^2([0,T_*];D^1_0),\  \sqrt{\rho}u_t\in L^\infty([0,T_*];L^2),\\
&  t^{\frac{1}{2}}u\in L^\infty([0,T_*];D^3),\ t^{\frac{1}{2}}\sqrt{\rho}u_{tt}\in L^2([0,T_*];L^2),\\
&t^{\frac{1}{2}}u_t\in L^\infty([0,T_*];D^1_0)\cap L^2([0,T_*];D^2),\ tu\in L^\infty([0,T_*]; D^{3,q}),\\
& tu_{tt}\in L^2([0,T_*];D^1_0),\ tu_t\in L^\infty([0,T_*]; D^{2}),\ t\sqrt{\rho}u_{tt}\in L^\infty([0,T_*];L^2).
\end{split}
\end{equation}
\underline{Step 2}: Uniqueness.
Let $(H_1,\rho_1,u_1)$ and $(H_2,\rho_2,u_2)$ be two solutions. Due to Lemma  \ref{lem1} in Section 3.1, we know $\rho_1=\rho_2$ and $H_1=H_2$.
  For the momentum equations $(\ref{eq:1.2rfv}) _4$, let $\overline{u}=u_1-u_2$, we have
 \begin{equation}\label{goplm}
\begin{split}
\rho \overline{u}_t-\mu\triangle \overline{u}-(\lambda+\mu) \nabla\text{div} \overline{u}=0,
\end{split}
\end{equation}
because we do not know  whether $\sqrt{\rho}\overline{u}\in L^\infty ([0,T_*];L^2(\Omega))$ or not, so we consider this equation in  bounded domain $\Omega_R$. We define
$\varphi^R(x)=\varphi(x/R)$,
where $\varphi\in C^\infty_c(B_1)$ is a smooth  cut-off function such that $\varphi=1$ in $ B_{1/2}$.
Let $\overline{u}^R=\varphi^R(t,x)u(t,x)$, we have
 \begin{equation}\label{goplm1}
\begin{split}
\rho \overline{u}^R_t-\mu\varphi^R\triangle \overline{u}-\varphi^R(\lambda+\mu) \nabla \text{div} \overline{u}=0.
\end{split}
\end{equation}
Therefore, multiplying (\ref{goplm1}) by $\overline{u}^R$ and integrating over $[0,t]\times \Omega_R$ ($t\in (0,T_*]$), we have
 \begin{equation}\label{goplm2}
\begin{split}
&\frac{1}{2}\int_{\Omega_R}\rho |\overline{u}^R|^2(t)\text{d}x+\int_0^t \int_{\Omega_R} \Big(\mu(\varphi^R)^2|\nabla \overline{u}|^2+(\lambda+\mu) (\varphi^R)^2|\text{div} \overline{u}|^2\Big)\text{d}x\text{d}s,\\
=&\int_0^t \int_{\Omega_R}\rho v \cdot \nabla \overline{u}^R \cdot \overline{u}^R \text{d}x\text{d}s-2\mu\int_0^t \int_{\Omega_R} \varphi^R (\overline{u} \cdot\nabla \overline{u})\cdot \nabla\varphi^R \text{d}x\text{d}s\\
&-2\int_0^t \int_{\Omega_R}(\lambda+\mu) \varphi^R\text{div}\overline{u} \nabla\varphi^R\cdot \overline{u} \text{d}x\text{d}s\doteq A_1+A_2+A_3.
\end{split}
\end{equation}
From Holder's inequality and Sobolev's imbedding theorem, we have
\begin{equation*}
\begin{split}
|A_1|\leq& \int_0^t \int_{\Omega_R } |\varphi^R\rho v \cdot  \nabla \overline{u}   \cdot  \overline{u}_R| \text{d}x\text{d}s+\int_0^t \int_{\Omega_R } |\rho\overline{u}_R|  |\nabla\varphi^R|v|\overline{u}| \text{d}x\text{d}s\\
\leq &C \int_0^t|\sqrt{\rho}\overline{u}_R|^2_2\text{d}s+\int_0^t\frac{\mu}{2}(\varphi^R)^2|\nabla \overline{u}|^2_2\text{d}s+\frac{C}{R^2}\int_0^t \int_{(\Omega_R \setminus B_{R/2})} |\overline{u}|^2  \text{d}x\text{d}s,\\
|A_2|\leq& \frac{C}{R}\int_0^t \int_{(\Omega_R \setminus B_{R/2})} |\overline{u}| |\nabla \overline{u}| \text{d}x\text{d}s\\
\leq &\frac{C}{R^2}\int_0^t \int_{(\Omega_R \setminus B_{R/2})} |\overline{u}|^2 \text{d}x\text{d}s+C\int_0^t \int_{(\Omega_R \setminus B_{R/2})} |\nabla \overline{u}|^2 \text{d}x\text{d}s\\
\leq & \frac{C}{R^2}\big|\Omega_R \setminus B_{R/2}\big|^{\frac{2}{3}} \int_0^t \Big(\int_{(\Omega_R \setminus B_{R/2})} |\overline{u}|^6 \text{d}x\Big)^{\frac{1}{3}}\text{d}s+C\int_0^t \int_{(\Omega_R \setminus B_{R/2})} |\nabla \overline{u}|^2 \text{d}x\text{d}s\\
\leq &C \int_0^{T_*} |\nabla \overline{u}(s)|^2_{L^2(\Omega_R \setminus B_{R/2})} \text{d}s\rightarrow 0\quad \text{as}\quad R\rightarrow \infty.
\end{split}
\end{equation*}
Similarly, we  can also obtain that
$$
|A_3|\leq C \int_0^{T_*} |\nabla \overline{u}(s)|^2_{L^2(\Omega_R \setminus B_{R/2})} \text{d}s\rightarrow 0\quad \text{as}\quad R\rightarrow \infty.
$$
Then from the above estimates, we deduce that
 \begin{equation}\label{kaka}
\begin{split}
&\frac{1}{2}\int_{\Omega_R}\rho |\overline{u}^R|^2(t)\text{d}x+\int_0^t \int_{\Omega_R} \mu(\varphi^R)^2|\nabla \overline{u}|^2\text{d}x\text{d}s\leq C \int_0^t|\sqrt{\rho}\overline{u}_R|^2_2\text{d}s+Q_R,
\end{split}
\end{equation}
where $
Q_R\rightarrow 0\quad  \text{as}\quad R\rightarrow \infty$.
Then letting $R\rightarrow \infty$ in (\ref{kaka}), via Gronwall's inequality,  we derive that $\overline{u}\equiv 0$, which means that $u_1=u_2$.

\underline{Step 3}: Time-continuity of the solution $(H,\rho, u,P)$. Firstly, the time-continuity of  $\rho$, $P$  and $H$ can be obtained by Lemma \ref{lem1}.  Secondly, from a classical embedding result (see \cite{gandi}), we have $ u \in C([0,T_*];D^1_0 )\cap  C([0,T_*];D^2-\textrm{weak})$.
From the momentum equations $(\ref{eq:1.2rfv})_4$, we know that
$(\rho u_t)_t\in L^2([0,T_*];H^{-1})$. Due to $\rho u_t\in L^2([0,T_*];D^1_0)$, we have immediately  that $\rho u_t\in C([0,T_*];D^1_0)$.
Similarly, from the following equations,
$$
Lu=-\rho u_t-\rho (v\cdot \nabla) v
  -\nabla P +\text{rot}H\times H\equiv F,
$$
where $F\in C([0,T_*];L^2)$, we can obtain $u\in C([0,T_*];D^2)$.
\end{proof}

\subsection{Proof of Theorem \ref{th5}}\ \\

Based on   Lemma \ref{lembbn}, now we give the proof of Theorem \ref{th5}. 
We first fix a positive constant $c_0$ sufficiently large such that
\begin{equation}\label{houmianq}\begin{split}
2+\overline{\rho}+\|(\rho_0-\overline{\rho},P_0-\overline{P},H_0)\|_{H^2\cap W^{2,q}}+|u_0|_{D^1_0\cap D^2}+|g_1|_{2}\leq c_0.
\end{split}
\end{equation}
Then let $u^0\in C([0,+\infty);D^1_0\cap D^2)\cap  L^{p_0}([0,+\infty);D^{3,q}) $ be the unique solution to the following linear parabolic problem
$$
h_t-\triangle h=0  \quad (0,+\infty)\times \Omega \quad \text{and} \quad h(0)=u_0 \quad \text{in} \quad \Omega.
$$
Then taking a small time $T^\epsilon\in (0,T_*]$, we have
\begin{equation*}\begin{split}
\sup_{0\leq t \leq T^\epsilon}|u^0(t)|^2_{D^1_0\cap D^2}+\int_{0}^{T^\epsilon} \Big( |u^0|^2_{D^3}+|u^0|^{p_0}_{ D^{3,q}}+|u^0_t|^2_{D^1_0}\Big)\text{d}t \leq& c_1,\\
\text{ess}\sup_{0\leq t \leq T^\epsilon}\Big(t|u^0_t(t)|^2_{D^1_0}+t|u^0(t)|^2_{D^3}\Big)+\int_{0}^{T^\epsilon} t|u^0_{t}|^2_{D^2}\text{d}t \leq& c_2,\\
\text{ess}\sup_{0\leq t \leq T^\epsilon}\Big(t^2|u^0(t)|^2_{ D^{3,q}}+t^2|u^0_t(t)|_{D^2}\Big)+\int_{0}^{T^\epsilon} t^2|u^0_{tt}|^2_{D^1_0}\text{d}t \leq& c_3
\end{split}
\end{equation*}
for constants $c_{i}'s$ with $1< c_0\leq c_1 \leq c_2 \leq  c_3 $.

\begin{proof}
From Lemma \ref{lembbn}, we know that there exists a unique classical solution $(H^1,\rho^1,P^1, u^1)$ to the linearized problem (\ref{eq:1.2rfv}) with $v$ replaced by $u^0$, which satisfies the estimate (\ref{hym}). Similarly, we construct approximate solutions $(H^{k+1},\rho^{k+1},P^{k+1}, u^{k+1})$ inductively, as follows: assuming that $u^{k}$ was defined for $k\geq 1$, let  $(H^{k+1},\rho^{k+1},P^{k+1}, u^{k+1})$  be the unique classical solutions to the problem (\ref{eq:1.2rfv})  with $v$ replaced by $u^{k}$ as following
 \begin{equation}
\label{eq:1.2rfvq}
\begin{cases}
\displaystyle
H^{k+1}_t+u^k\cdot \nabla H^{k+1}+(\text{div}u^kI_3-\nabla u^k)H^{k+1}=0, \\[6pt]
\displaystyle
\text{div}H^{k+1}=0, \\[6pt]
\displaystyle
\rho^{k+1}_t+\text{div}(\rho^{k+1} u^k)=0,  \\[6pt]
\displaystyle
\rho^{k+1} u^{k+1}_t+\rho^{k+1} u^k\cdot\nabla u^k
  +\nabla P^{k+1}+L^{k+1}u=\mu_0\text{rot}H^{k+1}\times H^{k+1}, \\[6pt]
(H^{k+1},\rho^{k+1},u^{k+1})|_{t=0}=(H_0(x),\rho_0(x),u_0(x))\quad x\in \Omega, \\[6pt]
(H^{k+1},\rho^{k+1},u^{k+1},P^{k+1})\rightarrow (0,\overline{\rho},0,\overline{P}) \quad \text{as } \quad |x|\rightarrow \infty,\quad t> 0.
\end{cases}
\end{equation}
Then  from Lemma \ref{lembbn} that $(H^k,\rho^k,P^k, u^k)$ satisfies  (\ref{hym}).
Next, we show that $(H^k, \rho^k,P^{k}, u^k)$ converges to a limit $(H,\rho,P,u)$ in a strong sense. But this can be done by a slight modification of the arguments in \cite{jishan}. We omits its details. Then adapting the proof of Lemma \ref{lembbn}, we can easily show that $(H,\rho,P,u)$ is a solution to  (\ref{eq:1.2})-(\ref{fan1}). The proof for uniqueness and time-continuity  is also similar to those in \cite{CK}\cite{jishan} and so omitted.
\end{proof}

\begin{remark}\label{qita}
For the case $0 < \sigma < +\infty$, if we add $H|_{\partial \Omega}=0$ to  (\ref{eq:1.2})-(\ref{fan1}), then the similar existence result can be obtained via the similar argument used in this Section.
\end{remark}

\section{Blow-up criterion for classical solutions}

 Now we prove (\ref{eq:2.91}). Let $(H, \rho, u)$ be the unique classical solution to   IBVP (\ref{eq:1.2})--(\ref{fan1}). We assume that the opposite holds, i.e.,
\begin{equation}\label{we11*}
\begin{split}
\lim \sup_{T\mapsto \overline{T}} |D( u)|_{L^1([0,T]; L^\infty(\Omega))}=C_0<\infty.
\end{split}
\end{equation}
Due to $P=A\rho^\gamma$, we quickly know that $P$  satisfies 
\begin{equation}\label{mou9}
P_t+u \nabla P+\gamma P \text{div}u=0, \quad P_0 \in H^2 \cap W^{2,q}.
\end{equation}

We first give  the standard energy estimate that 

  \begin{lemma}\label{s2}
\begin{equation*}
\begin{split}
\big(|\sqrt{\rho}u(t)|^2_{ 2}+|H|^2_2+|P|_1\big)+\int_{0}^{T}|\nabla u(t)|^2_{2}\text{d}t\leq C,\quad 0\leq t<  T,
\end{split}
\end{equation*}
where $C$ only depends on $C_0$ and $T$ $(any\  T\in (0,\overline{T}])$.
 \end{lemma}
\begin{proof} We first show that
\begin{align}
\label{2}\frac{d}{dt}\int_{\Omega} \big(\frac{1}{2}\rho |u|^2+\frac{P}{\gamma-1}+\frac{1}{2}H^2\big) \text{d}x+\int_{\Omega} \big(\mu |\nabla u|^2+(\lambda+\mu)(\text{div}u)^2\big)\text{d}x=0.
\end{align}

Actually,  (\ref{2})  is classical, which can be shown by multiplying   $(\ref{eq:1.2})_4$ by $u$,  $(\ref{eq:1.2})_3$ by $\frac{|u|^2}{2}$  and  $(\ref{eq:1.2})_1$ by $H$, then summing them together and integrating the result equation over $\Omega$ by parts, where we have used the fact
\begin{equation}\label{zhu1}
\begin{split}
\int_{\Omega} \text{rot}H \times H \cdot u \text{d}x=\int_{\Omega} -\text{rot}(u \times H)  \cdot H        \text{d}x.
\end{split}
\end{equation}
\end{proof}
 Let $f=(f^1,f^2,f^3)^\top\in \mathbb{R}^3$ and $g=(g^1,g^2,g^3)^\top\in \mathbb{R}^3$, we denote $(f\otimes g)_{ij}=(f_ig_j)$.
Next we need to show some lower order estimate for our classical solution $(H, \rho, u)$, which is the same as the regularity that the strong solution obtained in \cite{jishan} has to satisfy.
\subsection{Lower order estimate}\ \\

By assumption (\ref{we11*}), we first show that  both  $H$ and $\rho$ are both uniform bounded.
 \begin{lemma}\label{s1}
\begin{equation*}
\begin{split}
(|\rho(t)|_{\infty}+|H(t)|_{\infty}\big)\leq C, \quad 0\leq t< T,
\end{split}
\end{equation*}
where $C$ only depends on $C_0$ and $T$ $(any\  T\in (0,\overline{T}])$.
 \end{lemma}
 \begin{proof}

Multiplying $(\ref{eq:1.2})_1$ by $q |H|^{q-2} H$ and integrating  over $\Omega$ by parts, then we have
\begin{equation}\label{zhumw2}
\begin{split}
 \frac{d}{dt}|H|^q_q=& q \int_{\Omega} \big(H\cdot \nabla u-u\cdot \nabla H-H \text{div}u\big) \cdot H|H|^{q-2} \text{d}x\\
=&q \int_{\Omega} \big(H\cdot D( u)-u\cdot \nabla H-H \text{div}u\big) \cdot H|H|^{q-2} \text{d}x.
\end{split}
\end{equation}
By integrating by parts, the second term on the right-hand side can be written as
\begin{equation}\label{zhumw1}
\begin{split}
-q \int_{\Omega} \big(u\cdot \nabla H\big)\cdot H|H|^{q-2} \text{d}x=\int_{\Omega} \text{div}u |H|^q \text{d}x,
\end{split}
\end{equation}
which, together with (\ref{zhumw2}), immediately yields
\begin{equation}\label{zhumw3}
\begin{split}
& \frac{d}{dt}|H|^q_q\leq (2q+1) \int_{\Omega} |D( u) | |H|^{q} \text{d}x\leq (2q+1) |D( u)|_{\infty} |H|^q_q,
\end{split}
\end{equation}
which means that 
\begin{equation}\label{mou10}
 \frac{d}{dt}|H|_q\leq \frac{ (2q+1)}{q} |D( u)|_{\infty} |H|_q,
\end{equation}
hence, it follows from (\ref{we11*})  and (\ref{mou10}) that 
\begin{equation*}
\begin{split}
\sup_{0\leq t \leq T}|H|_{q}\leq C, \quad 0\leq T< \overline{T},
\end{split}
\end{equation*}
where $C>0$ is independent of $q$. Therefore, letting $q\rightarrow \infty$ in the above inequality leads to the desired estimate of $|H|_{\infty}$. In the same way, we also obtains the bound of $|\rho|_{\infty}$ which indeed depends only on $\|\text{div}u\|_{L^1([0,T];L^{\infty}(\Omega))}$.

 \end{proof}

The next lemma will give a key estimate on $\nabla H$, $\nabla \rho$ and $\nabla u$. 
  \begin{lemma}\label{s4}
\begin{equation*}
\begin{split}
\sup_{0\leq t\leq T}\big(|\nabla u|^2_{ 2}+|\nabla \rho|^2_{ 2}+|\nabla H|^2_2\big)+\int_0^T |\nabla^2 u|^2_2\text{d}t\leq C,\quad 0\leq T<  \overline{T},
\end{split}
\end{equation*}
where $C$ only depends on $C_0$ and $T$.
 \end{lemma}
\begin{proof}
%
Firstly, multiplying $(\ref{eq:1.2})_4$ by $\rho^{-1}\big(-Lu-\nabla P-\nabla |H|^2+H\cdot \nabla H\big) $ and integrating the result equation over $\Omega$, then we have
\begin{equation}\label{zhu6}
\begin{split}
&\frac{1}{2} \frac{d}{dt}\Big(\frac{\mu}{2}|\nabla u|^2_2+\frac{\mu+\lambda}{2}|\text{div}u|^2_2\Big)+\int_{\Omega}\rho^{-1}\big(-Lu-\nabla P-\nabla |H|^2+H\cdot \nabla H\big)^2\text{d}x\\
=&-\mu\int_{\Omega} (u\cdot \nabla u) \cdot \nabla \times (\text{rot}u)\text{d}x+(2\mu+\lambda)\int_{\Omega} (u\cdot \nabla u) \cdot  \nabla \text{div}u\text{d}x\\
&-\int_{\Omega} (u\cdot \nabla u) \cdot \nabla P(\rho) \text{d}x-\int_{\Omega} (u \cdot \nabla u) \big(\frac{1}{2} \nabla |H|^2-H\cdot \nabla H\big)\text{d}x\\
&-\int_{\Omega} u_t \cdot \nabla P(\rho) \text{d}x-\int_{\Omega} u_t\cdot  \big(\frac{1}{2} \nabla |H|^2-H\cdot \nabla H\big)\text{d}x\equiv: \sum_{i=1}^{6} L_i,
\end{split}
\end{equation}
where we have used the fact that $\triangle u=\nabla\text{div}u-\nabla\times \text{rot}u$.

We now estimate each term in (\ref{zhu6}). Due to the fact that $\rho^{-1}\geq C^{-1} >0$, we find the second term on the left hand side of (\ref{zhu6}) admits
\begin{equation}\label{gaibian}
\begin{split}
&\int_{\Omega}\rho^{-1}\big|Lu+\nabla P+\nabla |H|^2-H  \cdot \nabla H\big|^2\text{d}x\\
\geq& C^{-1}|Lu|^2_2-C(|\nabla P|^2_2+|\nabla u|^2_2+|H|^2_{\infty}|\nabla H|^2_2)\\
\geq& C^{-1}|u|^2_{D^2}-C(|\nabla \rho|^2_2+|\nabla u|^2_2+|\nabla H|^2_2),
\end{split}
\end{equation}
where we have used the standard $L^2$- theory of elliptic system and Lemma \ref{s1}. Note that $L$ is a strong elliptic operator. Next according to
\begin{equation*}
\begin{cases}
u\times \text{rot}u=\frac{1}{2}\nabla (| u|^2)-u \cdot \nabla u,\\[8pt]
\nabla\times(a\times b)=(b\cdot \nabla)a-(a \cdot \nabla)b+(\text{div}b)a-(\text{div}a)b,
\end{cases}
\end{equation*}
and Holder's inequality, Gagliardo-Nirenberg inequality and Young's inequality,
 we  deduce
\begin{equation}\label{zhu10cvcv}
\begin{split}
|L_1|=&\mu\Big|\int_{\Omega} (u \cdot \nabla u) \cdot \nabla \times (\text{rot} u)\text{d}x\Big|
=\mu\Big| \int_{\Omega} \nabla \times (u\cdot \nabla u) \cdot \text{rot} u\text{d}x\Big|\\
=&\mu\Big| \int_{\Omega} \nabla \times (u\times \text{rot}u) \cdot \text{rot} u\text{d}x\Big|\\
=&\mu\Big| \frac{1}{2}\int_{\Omega} (\text{rot} u)^2\text{div}u\text{d}x-\int_{\Omega} \text{rot} u\cdot D(u)\cdot \text{rot} u \text{d}x\Big|
\leq C|D(u)|_\infty|\nabla u|^2_2,
\end{split}
\end{equation}
\begin{equation}\label{zhu12aasd}
\begin{split}
|L_2|=&(2\mu+\lambda)\Big|\int_{\Omega}(u\cdot \nabla u) \cdot  \nabla \text{div}u\text{d}x\Big|\\
=&(2\mu+\lambda)\Big|-\int_{\Omega}\nabla u: (\nabla u)^\top \text{div}u\text{d}x+\frac{1}{2}\int_{\Omega} (\text{div}u)^3\text{d}x\Big|\\
\leq& C|D(u)|_\infty|\nabla u|^2_2,\\
L_3=&-\int_{\Omega} (u\cdot \nabla u) \cdot \nabla P \text{d}x\leq C|\nabla u|_2 |\nabla u|_{3}|\nabla P|_2\\
\leq& C(\epsilon)(|\nabla \rho|^2_2+1)|\nabla u|^2_2+\epsilon | u|^2_{D^2},\\
L_{4}=&-\int_{\Omega}  (u \cdot \nabla u) \big(\frac{1}{2} \nabla |H|^2-H\cdot \nabla H\big) \text{d}x\leq C|\nabla H|_2|H|_{\infty}|\nabla u|_3|u|_6\\
\leq& C(\epsilon)|H|^2_{\infty}|\nabla H|^2_2|\nabla u|^2_2+\epsilon \|\nabla u\|^2_1\leq C(\epsilon)(|\nabla H|^2_2+1)|\nabla u|^2_2+\epsilon | u|^2_{D^2},\\
L_{5}=&-\int_{\Omega} u_t \cdot \nabla P \text{d}x=\frac{d}{dt} \int_{\Omega}  P \text{div}u \text{d}x-\int_{\Omega}  P_t \text{div}u \text{d}x\\
=&\frac{d}{dt} \int_{\Omega}  P \text{div}u \text{d}x+\int_{\Omega} \big( u \cdot \nabla P \text{div}u +\gamma P (\text{div}u)^2\big) \text{d}x\\
\leq &\frac{d}{dt} \int_{\Omega}  P \text{div}u \text{d}x+C|\nabla P|_2 |u|_6 |\nabla u|_3+C|P|_\infty |\nabla u|^2_2\\
=&\frac{d}{dt} \int_{\Omega}  P \text{div}u \text{d}x+C(\epsilon)|\nabla u|^2_2(1+|\nabla \rho|^2_2)+\epsilon | u|^2_{D^2},\\
L_{6}=&-\int_{\Omega}  u_t \cdot \big(\frac{1}{2} \nabla |H|^2-H\cdot \nabla H\big)\text{d}x\\
=& \frac{1}{2}\frac{d}{dt}\int_{\Omega} |H|^2  \text{div}u \text{d}x-\frac{d}{dt}\int_{\Omega} H \cdot \nabla u \cdot H \text{d}x\\
&-\int_{\Omega}  \text{div}u  H \cdot H_t \text{d}x+\int_{\Omega} H_t \cdot \nabla u \cdot H \text{d}x+\int_{\Omega} H\cdot \nabla u \cdot H_t \text{d}x.
\end{split}
\end{equation}
where  we have used the fact $\text{div}H=0$ and $\epsilon> 0$ is a sufficiently small constant.  To deal with the last three terms on the right-hand side of $L_6$, we need to use 
$$
H_t=H \cdot \nabla u-u \cdot \nabla H-H\text{div}u.$$
Hence, similar to the proof of the above estimates for $L_i$, we also have
\begin{equation}\label{zhu12vvvv}
\begin{split}
\quad \quad \int_{\Omega}   H_t \cdot \nabla u \cdot H \text{d}x=&\int_{\Omega} - \text{div}u  H \cdot \big(H \cdot \nabla u-u \cdot \nabla H-H\text{div}u\big) \text{d}x\\
\leq & C|H|^2_\infty |\nabla u|^2_2+|D(u)|_{\infty} |\nabla H|_2 |u|_6 |H|_3\\
\leq & C(|D(u)|_{\infty}+1)(|\nabla u|^2_2+|\nabla H|^2_2),\\
\int_{\Omega} H_t \cdot \nabla u \cdot H \text{d}x+&\int_{\Omega} H\cdot \nabla u \cdot H_t \text{d}x\\
\leq &\int_{\Omega} |(H \cdot \nabla u-u \cdot \nabla H-H\text{div}u\big) \cdot \nabla u \cdot H| \text{d}x\\
\leq & C|H|^2_\infty |\nabla u|^2_2+|u|_{\infty} |\nabla u|_2 |\nabla H|_2 |H|_\infty\\
\leq & C(\epsilon)(|\nabla H|_{2}+1)|\nabla u|^2_2+\epsilon |u|^2_{D^2}.
\end{split}
\end{equation}

Then combining (\ref{zhu6})-(\ref{zhu12vvvv}),  we have
\begin{equation}\label{zhu6qss}
\begin{split}
&\frac{1}{2} \frac{d}{dt}\int_{\Omega}(\mu|\nabla u|^2+(\mu+\lambda)|\text{div}u|^2-(P+\frac{1}{2}|H|^2) \text{div}u H\cdot \nabla u\cdot H\Big)\text{d}x+C|\nabla^2 u|^2_2\\
\leq &C(|\nabla u|^2_2+|\nabla H|^2_2+1)(|\nabla u|^2_2+|D(u)|_\infty+1).
\end{split}
\end{equation}

Secondly, applying $\nabla$ to  $(\ref{eq:1.2})_3$ and multiplying the result equation by $2\nabla \rho$,  we have
\begin{equation}\label{zhu20}
\begin{split}
&(|\nabla \rho|^2)_t+\text{div}(|\nabla \rho|^2u)+|\nabla \rho|^2\text{div}u\\
=&-2 (\nabla \rho)^\top \nabla u \nabla \rho-2 \rho \nabla \rho \cdot \nabla \text{div}u\\
=&-2 (\nabla \rho)^\top D(u) \nabla \rho-2 \rho \nabla \rho \cdot \nabla \text{div}u.
\end{split}
\end{equation}
Then integrating (\ref{zhu20}) over $\Omega$, we have
\begin{equation}\label{zhu21}
\begin{split}
\frac{d}{dt}|\nabla \rho|^2_2
\leq& C(|D( u)|_\infty+1)|\nabla \rho|^2_2+\epsilon |\nabla^2 u|^2_2.
\end{split}
\end{equation}

Thirdly, applying $\nabla $  to $(\ref{eq:1.2})_1$, due to 
\begin{equation}\label{mou6}
\begin{split}
A=\nabla (H\cdot \nabla u)=&(\partial_j H \cdot \nabla u^i)_{(ij)}+(H\cdot \nabla \partial_j u^i)_{(ij)},\\
B=\nabla (u \cdot \nabla H)=&(\partial_j u\cdot \nabla H^i)_{(ij)}+ (u\cdot \nabla \partial_j H^i)_{(ij)},\\
C=\nabla(H \text{div}u)=&\nabla H \text{div}u+H \otimes \nabla \text{div}u,
\end{split}
\end{equation}
then  multiplying the result equation  $\nabla (\ref{eq:1.2})_1$ by $2\nabla H$, we have
\begin{equation}\label{zhu20q}
\begin{split}
&(|\nabla H|^2)_t-2A:\nabla H+2 B \nabla H-2C : \nabla H=0.
\end{split}
\end{equation}

Then integrating (\ref{zhu20q}) over $\Omega$, due to
\begin{equation}\label{zhu20qq}
\begin{split}
&\int_{\Omega}  A: \nabla H \text{dx}\\
=&\int_{\Omega}  \sum_{j=1}^3\Big(\sum_{i=1}^3 \sum_{k=1}^3  \partial_j H^k \partial_k u^i \partial_j H^i\Big) \text{dx}+\int_{\Omega}  \sum_{j=1}^3\sum_{i=1}^3 \sum_{k=1}^3  H^k \partial_{kj} u^i \partial_j H^i \text{dx}\\
=&\int_{\Omega}  \sum_{j=1}^3\Big(\sum_{i,k} \partial_j H^k \frac{(\partial_k u^i +\partial_i u^k)}{2}\partial_j H^i\Big) \text{dx}+\int_{\Omega}  \sum_{j=1}^3\sum_{i=1}^3 \sum_{k=1}^3  H^k \partial_{kj} u^i \partial_j H^i \text{dx}\\
\leq& C|D(u)|_\infty |\nabla H|^2_2+C|H|_\infty|\nabla H|_2 |u|_{D^2},\\
&\int_{\Omega}  B: \nabla H \text{dx}\\
=&\int_{\Omega}  \sum_{j=1}^3\sum_{i=1}^3 \sum_{k=1}^3  \partial_j u^k \partial_k H^i \partial_j H^i \text{dx}+\int_{\Omega}  \sum_{j=1}^3\sum_{i=1}^3 \sum_{k=1}^3  u^k  \partial_{kj} H^i \partial_j H^i \text{dx}\\
=&\int_{\Omega}  \sum_{i=1}^3\Big(\sum_{j,k}   \partial_k H^i \frac{(\partial_j u^k+\partial_k u^j)}{2} \partial_j H^i \Big)\text{dx}+\frac{1}{2}\int_{\Omega}  \sum_{i=1}^3\Big(\sum_{j,k}  u^k  \partial_{k} ( \partial_j H^i)^2\Big) \text{dx}\\
\leq& C|D(u)|_\infty |\nabla H|^2_2,
\end{split}
\end{equation}
\begin{equation}\label{zhu20qbbq}
\begin{split}
\int_{\Omega}  C: \nabla H \text{dx}
=&\int_{\Omega} \big(\text{div}u|\nabla H|^2+(H\otimes \nabla \text{div}u) :\nabla H\big) \text{dx}\\
\leq& C|D(u)|_\infty |\nabla H|^2_2+C|H|_\infty|\nabla H|_2 |u|_{D^2},
\end{split}
\end{equation}
 we quickly have the following estimate from (\ref{zhu20q})-(\ref{zhu20qbbq}):
\begin{equation}\label{zhu21q}
\begin{split}
\frac{d}{dt}|\nabla H|^2_2
\leq& C(|D( u)|_\infty+1)|\nabla H|^2_2+\epsilon |\nabla^2 u|^2_2.
\end{split}
\end{equation}

Adding (\ref{zhu21}) and (\ref{zhu21q}) to (\ref{zhu6qss}), from Gronwall's inequality we immediately obtain
\begin{equation*}
\begin{split}
|\nabla u(t)|^2_{ 2}+|\nabla \rho(t)|^2_{ 2}+|\nabla H(t)|^2_{ 2}+\int_0^t |\nabla^2 u(s)|^2_2\text{d}t\leq C,\quad 0\leq t< T.
\end{split}
\end{equation*}

 \end{proof}

Next, we proceed to improve the regularity of $\rho$, $H$ and $u$. To this end, we first drive some bounds on derivatives of $u$ based on estimates above.
Now we give the  estimates for the lower order terms of the velocity $u$.

 \begin{lemma}[\textbf{Lower order estimate of the velocity $u$}]\label{wlem:4-1}\ \\ 
\begin{equation*}
\begin{split}
|u(t)|^2_{  D^2}+|\sqrt{\rho} u_t(t)|^2_{2} + \int_{0}^{T} |u_t|^2_{D^1}\text{d}t\leq C, \quad 0\leq t \leq T,
\end{split}
\end{equation*}
where $C$ only depends on $C_0$ and $T$ $(any\  T\in (0,\overline{T}])$.

 \end{lemma}

 \begin{proof}

Via  $(\ref{eq:1.2})_4$ and Lemmas \ref{ok}, \ref{s2}-\ref{s4},  we show that 
\begin{equation}\label{mousun}
\begin{split}
|u|_{D^2}\leq C(|\sqrt{\rho}u_t|_2+1).
\end{split}
\end{equation}

 Differentiating $ (\ref{eq:1.2})_4$ with respect to $t$, we have
\begin{equation}\label{wgh78}
\begin{split}
\rho u_{tt}+Lu_t=-\rho_tu_t -\rho u\cdot\nabla u_t-\rho_t u\cdot\nabla u-\rho u_t\cdot\nabla u-\nabla P_t
 +(\text{rot}H\times H)_t.
\end{split}
\end{equation}
Multiplying (\ref{wgh78}) by $u_t$ and integrating  over $\Omega$, we have
\begin{equation}\label{wzhen4}
\begin{split}
&\frac{1}{2}\frac{d}{dt}\int_{\Omega}\rho |u_t|^2 \text{d}x+\int_{\Omega}(\mu|\nabla u_t|^2+(\lambda+\mu)(\text{div}u_t)^2) \text{d}x\\
=&  -\int_{\Omega} \rho u \cdot \nabla  |u_t|^2\text{d}x-\int_{\Omega} \rho u \nabla ( u \cdot \nabla u \cdot u_t)\text{d}x-\int_{\Omega} \rho u_t \cdot \nabla u \cdot u_t\text{d}x+\int_{\Omega} P_t \text{div} u_t\text{d}x\\
& +\int_{\Omega}H  \cdot H_t \text{div} u_t\text{d}x-\int_{\Omega}\big(H \cdot \nabla u_t \cdot H_t+H_t \nabla u_t \cdot H\big)\text{d}x
\equiv:\sum_{i=7}^{12}L_i,
\end{split}
\end{equation}
where we have used the fact $\text{div}H=0$.

According to Lemmas \ref{s2}-\ref{s4}, Holder's inequality, Gagliardo-Nirenberg inequality and Young's inequality,  we deduce that
\begin{equation}\label{wzhou6}
\begin{split}
L_7=&-\int_{\Omega} \rho u \cdot \nabla  |u_t|^2\text{d}x\leq  C|\rho|^{\frac{1}{2}}_{\infty}|u|_{\infty}|\sqrt{\rho} u_t|_{2}|\nabla u_t|_{2}\leq C\|\nabla u\|^2_1|\sqrt{\rho} u_t|^2_{2}+\epsilon |\nabla  u_t|^2_2, \\
L_8=& -\int_{\Omega} \rho u \nabla ( u \cdot \nabla u \cdot u_t)\text{d}x
 \leq C\int_{\Omega} \big(| u|  |\nabla u|^2 |u_t|+| u|^2  |\nabla^2 u| |u_t|+| u|^2  |\nabla u| |\nabla u_t| \big)\text{d}x\\
\leq& C |u_t|_6 ||\nabla u|^2|_{\frac{3}{2}} |u|_{6}+C ||u|^2|_{3} |\nabla^2 u|_{2} |u_t|_{6}+C||u|^2|_{3} |\nabla u|_{6}  | \nabla u_t|_{2}\\
\leq& C \big( |\nabla u|^2_{3} |\nabla u|_{2}+ |\nabla u|^2_{2} \|\nabla u\|_{1}   \big)|\nabla u_t|_{2}\\
\leq & C\|\nabla u\|_{1}| \nabla u_t|_{2} \leq \epsilon|\nabla u_t|^2_{2}+C(\epsilon)\|\nabla u\|^2_{1},
\end{split}
\end{equation}
where we have used the fact that 
\begin{equation}\label{mou5}
||u|^2|_{3} \leq C|u|^2_{6}  \leq C |\nabla u|^2_{2},\quad |\nabla u|^2_{3} \leq C|\nabla u|_{2} |\nabla u|_{6} \leq C |\nabla u|_{2} \|\nabla u\|_{1}.
\end{equation}
And similarly, we also have
\begin{equation}\label{wmou4}
\begin{split}
L_9=&-\int_{\Omega} \rho u_t \cdot \nabla u \cdot u_t\text{d}x  \leq  C|\rho|^{\frac{1}{2}}_{\infty}|u_t|_{6}|\sqrt{\rho} u_t|_{2}|\nabla u|_{3}\\
\leq& \epsilon|\nabla u_t|^2_{2}+C(\epsilon)|\sqrt{\rho} u_t|^2_{2} \|\nabla u\|^2_{1},\\
L_{10}=&\int_{\Omega} P_t \text{div} u_t\text{d}x \leq \int_{\Omega}|u\cdot \nabla P+\gamma P\text{div}v| |\nabla u_t| \text{d}x\\
\leq &C|u|_{\infty}|\nabla P|_{2}|\nabla u_t|_{2}+C|P|_{\infty}|\text{div} u|_{2}|\nabla u_t|_{2}\\
\leq & \epsilon|\nabla u_t|^2_{2}+C(\epsilon)\|\nabla u\|^2_{1},\\
L_{11}+L_{12}=&
\int_{\Omega}H  \cdot H_t \text{div} u_t\text{d}x-\int_{\Omega}\big(H \cdot \nabla u_t \cdot H_t+H_t \nabla u_t \cdot H\big)\text{d}x\\
\leq& C|\nabla u_t|_2 |H_t|_2 |H|_{\infty}\leq  C\big(|H|_\infty |\nabla u|_2+|u|_\infty |\nabla H|_2\big) |\nabla u_t|_{2}\\
\leq &  \epsilon|\nabla u_t|^2_{2}+C(\epsilon)\|\nabla u\|^2_{1}.
\end{split}
\end{equation}

Then combining the above estimate (\ref{wzhou6})-(\ref{wmou4}),  from (\ref{wzhen4}), we have
\begin{equation}\label{wzhen5g}
\begin{split}
&\frac{1}{2}\frac{d}{dt}\int_{\Omega}\rho |u_t|^2 \text{d}x+\int_{\Omega}|\nabla u_t|^2\text{d}x
\leq C(|\sqrt{\rho}u_t|^2_{2}+1)(\|\nabla u\|^2_1+1).
\end{split}
\end{equation}
Then integrating (\ref{wzhen5g}) over $(\tau,t)$ ($\tau\in (0,t)$), for $\tau\leq t \leq T$, we  have
\begin{equation}\label{wnvk3}
\begin{split}
&|\sqrt{\rho}u_t(t)|^2_2 +\int_{\tau}^{t}|\nabla u_t|^2_{D^1}\text{d}s
\leq |\sqrt{\rho}u_t(\tau)|^2_2+C\int_{\tau}^{t}(\|\nabla u\|^2_1+1) |\sqrt{\rho}u_t|^2_2 \text{d}s+C.
\end{split}
\end{equation}
From the momentum equations  $ (\ref{eq:1.2})_4$, we easily have
\begin{equation}\label{wli9}
\begin{split}
|\sqrt{\rho}u_t(\tau)|^2_2\leq C\int_{\Omega} \rho |u|^2|\nabla u|^2\text{d}x+C\int_{\Omega} \frac{|\nabla P+Lu- \text{rot}H\times H|^2}{\rho}\text{d}x,
\end{split}
\end{equation}
due to the initial layer compatibility condition (\ref{th79}), letting $\tau\rightarrow 0$ in (\ref{wli9}), we have
\begin{equation}\label{wnvk33}
\begin{split}
\lim \sup_{\tau\rightarrow 0}|\sqrt{\rho}u_t(\tau)|^2_2 \leq C\int_{\Omega} \rho_0 |u_0|^2|\nabla u_0|^2\text{d}x+C\int_{\Omega} |g_1|^2\text{d}x\leq C.
\end{split}
\end{equation}
Then, letting $\tau\rightarrow 0$ in (\ref{wnvk3}), from Gronwall's inequality and (\ref{mousun}), we deduce that
\begin{equation}\label{wnvk55}
\begin{split}
|\sqrt{\rho}u_t(t)|^2_2 +|u(t)|_{D^2}+\int_{0}^{t}|\nabla u_t|^2_{D^1}\text{d}s\leq C,\ 0 \leq t \leq T.
\end{split}
\end{equation}

\end{proof}

Finally, the following lemma gives bounds of $\nabla \rho$, $\nabla H$ and $\nabla^2 u$.
 \begin{lemma}\label{s7}
 \begin{equation}\label{zhu54}
\begin{split}
&\big(\|\big(\rho,H,P)(t)\|_{W^{1,q}}+|(\rho_t,H_t, P_t)(t)|_q\big)+\int_0^T|u(t)|^2_{D^{2,q}}\text{d}t\leq C,
\quad 0\leq t<  T,
\end{split}
\end{equation}
where $C$ only depends on $C_0$ and $T$ $(any \ T \in (0,\overline{T}])$, and $q\in (3,6]$.
\end{lemma}
\begin{proof}
Via  $(\ref{eq:1.2})_4$ and Lemmas \ref{ok}, \ref{s2}-\ref{wlem:4-1},  we show that 
 \begin{equation}\label{zhu55}
\begin{split}
|\nabla^2 u|_q \leq& C(|\rho u_t|_q+| \rho u\cdot \nabla u|_q+|\nabla P|_q+|\text{rot}H \times H|_q+|u|_{D^{1,q}_0})\\
\leq& C(1+|\nabla u_t|_2+|\nabla P|_q+  |\nabla H|_q).
\end{split}
\end{equation}

Firstly,  applying $\nabla$ to  $(\ref{eq:1.2})_3$, multiplying the result equations by $q|\nabla \rho|^{q-2} \nabla \rho$, we have
\begin{equation}\label{zhu20cccc}
\begin{split}
&(|\nabla \rho|^q)_t+\text{div}(|\nabla \rho|^qu)+(q-1)|\nabla \rho|^q\text{div}u\\
=&-q |\nabla \rho|^{q-2}(\nabla \rho)^\top D( u) (\nabla \rho)-q \rho|\nabla \rho|^{q-2} \nabla \rho \cdot \nabla \text{div}u.
\end{split}
\end{equation}
Then integrating (\ref{zhu20cccc}) over $\Omega$, we immediately obtain
\begin{equation}\label{zhu200}
\begin{split}
\frac{d}{dt}|\nabla \rho|_q
\leq& C|D( u)|_\infty|\nabla \rho|_q+C|\nabla^2 u|_q.
\end{split}
\end{equation}

Secondly, applying $\nabla$ to  $(\ref{eq:1.2})_1$, multiplying the result equations by $q\nabla H |\nabla H|^{q-2}$, we have
\begin{equation}\label{zhu20qs}
\begin{split}
&(|\nabla H|^2)_t-qA:\nabla H|\nabla H|^{q-2}+q B \nabla H|\nabla H|^{q-2}+qC : \nabla H|\nabla H|^{q-2}=0.
\end{split}
\end{equation}

Then integrating (\ref{zhu20qs}) over $\Omega$, due to
\begin{equation}\label{zhu20qqs}
\begin{split}
&\int_{\Omega}  A: \nabla H |\nabla H|^{q-2}\text{dx}\\
=&\int_{\Omega}  \sum_{j=1}^3\Big(\sum_{i,k}  \partial_j H^k \partial_k u^i \partial_j H^i \Big)|\nabla H|^{q-2} \text{dx}+\int_{\Omega}  \sum_{j=1}^3\sum_{i=1}^3 \sum_{k=1}^3  H^k \partial_{kj} u^i \partial_j H^i |\nabla H|^{q-2}\text{dx}\\
\leq& C|D(u)|_\infty |\nabla H|^q_q+C|H|_\infty|\nabla H|^{q-1}_q |u|_{D^{2,q}},
\end{split}
\end{equation}
\begin{equation}\label{zhu20qqs2}
\begin{split}
&\int_{\Omega}  B: \nabla H|\nabla H|^{q-2} \text{dx}\\
=&\int_{\Omega}  \sum_{j=1}^3\sum_{i=1}^3 \sum_{k=1}^3  \partial_j u^k \partial_k H^i \partial_j H^i|\nabla H|^{q-2} \text{dx}+\int_{\Omega}  \sum_{j=1}^3\sum_{i=1}^3 \sum_{k=1}^3  u^k  \partial_{kj} H^i \partial_j H^i|\nabla H|^{q-2} \text{dx}\\
=&\int_{\Omega}  \sum_{i=1}^3\Big(\sum_{j,k}  \partial_j u^k \partial_k H^i \partial_j H^i \Big)|\nabla H|^{q-2} \text{dx}+\frac{1}{2}\int_{\Omega}  \sum_{k=1}^3  u^k\Big(\sum_{j,i}   \partial_{k} |\partial_j H^i|^2|\nabla H|^{q-2} \Big)\text{dx}\\
=&\int_{\Omega}  \sum_{i=1}^3\Big(\sum_{j,k} \partial_k H^i    \partial_j u^k \partial_j H^i \Big)|\nabla H|^{q-2} \text{dx}+\frac{1}{2}\int_{\Omega}  \sum_{k=1}^3  u^k\Big(\sum_{j,i} \partial_{k} |\nabla H|^2|\nabla H|^{q-2} \Big)\text{dx}\\
=&\int_{\Omega}  \sum_{i=1}^3\Big(\sum_{j,k}   \partial_k H^i \partial_j u^k \partial_j H^i \Big)|\nabla H|^{q-2} \text{dx}+\frac{1}{q}\int_{\Omega}  \sum_{k=1}^3  u^k  \partial_{k} |\nabla H|^{q} \text{dx}\\
\leq& C|D(u)|_\infty |\nabla H|^q_q,\\
&\int_{\Omega}  C: \nabla H|\nabla H|^{q-2} \text{dx}
=\int_{\Omega} \big(\text{div}u|\nabla H|^q+(H\otimes \nabla \text{div}u) :\nabla H|\nabla H|^{q-2}\big) \text{dx}\\
\leq& C|D(u)|_\infty |\nabla H|^q_q+C|H|_\infty|\nabla H|^{q-1}_q |u|_{D^{2,q}},
\end{split}
\end{equation}
 we quickly obtain the following estimate:
\begin{equation}\label{zhu21qs}
\begin{split}
\frac{d}{dt}|\nabla H|_q
\leq& C(|D( u)|_\infty+1)|\nabla H|_q+C | u|_{D^{2,q}}.
\end{split}
\end{equation}

Then from  (\ref{zhu55}), (\ref{zhu200}),  (\ref{zhu21qs})  and Gronwall's inequality, we immediately have
$$
(|\nabla \rho(t)|_q+|\nabla H(t)|_q)\leq C\exp \Big(\int_0^t(1+|D( u)|_\infty) \text{d}s\Big)\leq C, \quad 0\leq t \leq T.
$$

Finally, via (\ref{zhu55}) and Lemma \ref{wlem:4-1}, we easily have
\begin{equation}\label{zhu25ss}
\begin{split}
\int_0^t|u(s)|^2_{D^{2,q}}\text{d}s
\leq& C\int_0^t(1+ |\nabla u_t(s)|^2_2)\text{d}s\leq C,\quad 0\leq t\leq T.
\end{split}
\end{equation}
 \end{proof}

\subsection{Improved regularity}\ \\
In this section, we will get some higher order regularity of $H$, $\rho$ and $u$ to make sure that this solution is a classical one in $[0,\overline{T}]$.  Based on the estimates obtained in the above section, in truth, we have already proved that $\int_0^t|\nabla u|^2_\infty \text{d}s\leq C$. 
\begin{lemma}[\textbf{Higher order estimate }]\label{s8}
\begin{equation*}
\begin{split}
& \big(|(\rho,P,H)(t)|^2_{D^2}+\|(\rho_t,P_t,H_t)(t)\|^2_1\big)+\int_{0}^{T}\Big(|u|^2_{D^{3}}+|(\rho_{tt},P_{tt},H_{tt})|^2_{2}\Big)\text{d}t\leq C, \quad 0\leq t<  T,
\end{split}
\end{equation*}
where $C$ only depends on $C_0$ and $T$ $(any \ T \in (0,\overline{T}])$.
 \end{lemma}
 \begin{proof}Via  $(\ref{eq:1.2})_4$ and Lemmas \ref{ok}, \ref{s2}-\ref{s7},  we show that 
 \begin{equation}\label{zhu150}
\begin{split}
|u|_{D^3}\leq& C(|\rho u_t|_{D^1}+|\rho u\cdot \nabla u|_{D^1}+|\nabla P|_{D^1}+|\text{rot}H \times H|_{D^1})\\
\leq& C(1+|u_t|_{D^1}+| P|_{D^2}+|H|_{D^2}).
\end{split}
\end{equation}

Firstly,
 applying $\nabla^2$  to  $(\ref{eq:1.2})_3$, and 
multiplying the result equation by $2\nabla^2 \rho$, integrating  over $\Omega$ we easily deduce that
\begin{equation}\label{zhenzhen11}\begin{split}
\frac{d}{dt}|\rho|^2_{D^2}\leq& C|\nabla u|_\infty|\rho|^2_{D^2}+C|\rho|_{\infty}|u|_{D^3}|\rho|_{D^2}+|\nabla \rho|_3|\nabla^2 \rho|_2\|\nabla ^2 u\|_1,
\end{split}
\end{equation}
which, together with (\ref{zhu150}),
\begin{equation}\label{zhenzhenss}\begin{split}
\frac{d}{dt}|\rho|_{D^2}\leq& C(|\nabla u|_\infty+1)(1+|\rho|_{D^2}+|P|_{D^2}+|H|_{D^2})+C|\nabla u_t|^2_2.
\end{split}
\end{equation}
And similarly, we have 
\begin{equation}\label{zhenzhensss}\begin{cases}
\displaystyle
\frac{d}{dt}|H|_{D^2}\leq C(|\nabla u|_\infty+1)(1+|P|_{D^2}+|H|_{D^2})+C|\nabla u_t|^2_2,\\[8pt]
\displaystyle
\frac{d}{dt}|P|_{D^2}\leq C(|\nabla u|^2_\infty+1)(1+|P|_{D^2}+|H|_{D^2})+C|\nabla u_t|^2_2.
\end{cases}
\end{equation}

So combining (\ref{zhenzhenss})- (\ref{zhenzhensss}),  we quickly have
\begin{equation}\label{zhmm12acc}
\begin{split}
&\frac{d}{dt}(|\rho|_{D^2}+|H|_{D^2}+|P|_{D^2})\\
\leq& C(1+|\nabla u|_\infty)(|\rho|_{D^2}+|H|_{D^2}+|P|_{D^2})+C(1+|\nabla u_t|^2_2).\end{split}
\end{equation}
Then via Gronwall's inequality and (\ref{zhmm12acc}),  we obtain
\begin{equation*}
\begin{split}
|\rho|_{D^2}+|H|_{D^2}+|P|_{D^2}+\int_{0}^{t}|u(s)|^2_{D^{3}}\text{d}t\leq C, \quad 0\leq t\leq T.
\end{split}
\end{equation*}

Finally, due to the following relation
\begin{equation}\label{ghtuss}
\begin{cases}
H_t=H \cdot \nabla u-u \cdot \nabla H-H\text{div}u,\\[6pt]
\rho_t=-u\cdot \nabla \rho-\rho\text{div} u,\ P_t=-u \cdot \nabla P-\gamma P \text{div}u,
\end{cases}
\end{equation}
we immediately get the desired conclusions.
\end{proof}
Now we will give some estimates for the higher order terms of the velocity $u$ in the following three Lemmas.
\begin{lemma}[\textbf{Higher order estimate of the velocity $u$}]\label{wlem:4-3}\ \\
\begin{equation*}
\begin{split}
t|u_t(t)|^2_{D^1_0}+t|u(t)|^2_{D^3}+\int_{0}^{T}t\big(|u_t|^2_{D^{2}}+|\sqrt{\rho} u_{tt}|^2_{2}\big)\text{d}s\leq C, \quad 0\leq t \leq T,
\end{split}
\end{equation*}
where $C$ only depends on $C_0$ and $T$ $(any \ T \in (0,\overline{T}])$.
 \end{lemma}
\begin{proof}
Firstly, multiplying  (\ref{wgh78}) by $u_{tt}$ and integrating  over $\Omega$, we have
\begin{equation}\label{wf3}
\begin{split}
&\int_{\Omega}\rho |u_{tt}|^2 \text{d}x+\frac{1}{2}\frac{d}{dt}\int_{\Omega}\Big(\mu|\nabla u_t|^2+(\lambda+\mu)(\text{div}u_t)^2\Big) \text{d}x\\
=& \int_{\Omega} \Big(\big(-\nabla P_t-(\rho u\cdot \nabla u)_t-\rho_tu_t+(\text{rot}H\times H)_t\big)\cdot u_{tt}\Big)\text{d}x
=\frac{d}{dt}\Phi_1(t)+ \Phi_2(t),
\end{split}
\end{equation}
where
\begin{equation*}
\begin{split}
\Phi_1(t)=&\int_{\Omega} \big( P_t\text{div}u_t-\rho_t(u\cdot \nabla u)\cdot u_t-\frac{1}{2}\rho_t|u_t|^2+(\text{rot}H\times H)_t \cdot u_t \big)\text{d}x,\\
\Phi_2(t)=&\int_{\Omega} \big(-P_{tt}\text{div}u_t-\rho (u\cdot \nabla u)_t \cdot u_{tt}+\rho_{tt} (u\cdot \nabla u) \cdot u_t
 +\rho_{t} (u\cdot \nabla u)_t \cdot u_t\big)\text{d}x\\
 &+\int_{\Omega}\big( \frac{1}{2}\rho_{tt}|u_t|^2- (\text{rot}H\times H)_{tt}\cdot u_t\big)\text{d}x\equiv:\sum_{i=13}^{18}L_i.
\end{split}
\end{equation*}
Then  almost same to (\ref{wzhou6}), we also have
\begin{equation}\label{wzhen6}\begin{split}
\Phi_1(t)\leq \frac{\mu}{10} |\nabla u_t|^2_2+C.
\end{split}
\end{equation}
Let we denote
$$
\Phi^*(t)=\frac{1}{2}\int_{\Omega}\mu|\nabla u_t|^2+(\lambda+\mu)(\text{div}u_t)^2\text{d}x-\Lambda_3(t),
$$
then from  (\ref{wzhen6}), for $0 \leq t \leq T$, we quickly  have
\begin{equation}\label{wkui}
\begin{split}
C|\nabla u_t|^2_{2}-C \leq \Phi^*(t)\leq&  C|\nabla u_t|^2_{2}+C.
\end{split}
\end{equation}

Similarly, according to Lemmas \ref{s1}-\ref{s8}, Holder's inequality and  Gagliardo-Nirenberg inequality, for $0 < t \leq T$,  we deduce that
\begin{equation}\label{wght}
\begin{split}
&L_{13}\leq C|P_{tt}|_{2}  | \nabla u_t|_{2},\ L_{14} \leq |\rho|^{\frac{1}{2}}_{\infty} |\sqrt{\rho}u_{tt}|_2\big(|u|_{\infty} |\nabla u_t|_2+|\nabla u|_3 |\nabla u_t|_2\big),\\
&L_{15}\leq C|\rho_{tt}|_2 |\nabla u_t|_2 |\nabla u|_3 |u|_{\infty}, \\
& L_{16}\leq C|\rho_t|_2 |u_{t}|_6 |\nabla u|_{6} |\nabla u_t|_2+C|u|_{\infty}|u_t|_{6}|\nabla u_t|_{2}|\rho_t|_3,\\
&L_{17}\leq C|\rho_{t}|_3 |\nabla u_t|_2|u|_{\infty}|u_t|_{6}+C|\rho|^{\frac{1}{2}}_{\infty} |\sqrt{\rho}u_t|_{3} | u_t|_{6} |\nabla u_t|_2,
\end{split}
\end{equation}
where we have used the facts $\rho_t=-\text{div}(\rho u)$,
and
\begin{equation}\label{wzhen7}
\begin{split}
 L_{18}=&-\int_{\Omega}(\text{rot}H\times H)_{tt}\cdot u_{t}\text{d}x
=\int_{\Omega}\big(H\otimes H-\frac{1}{2}|H|^2I_3\big)_{tt}: \nabla u_{t}\text{d}x\\
\leq& C|\nabla u_t|_2|H_t|^2_4+C|\nabla u_t|_2|H_{tt}|_2| H|_{\infty}.
\end{split}
\end{equation}
Combining (\ref{wght})-(\ref{wzhen7}), from Young's inequality, we have
\begin{equation}\label{wzhen10}\begin{split}
\Phi_2(t) \leq&\frac{1}{2}|\sqrt{\rho} u_{tt}(t)|^2_{2}+C(1+|\nabla u_t|^2_{2})|\nabla u_t|^2_{2} 
C|\nabla u|^2_{\infty}
+C(|P_{tt}|^2_2+|\rho_{tt}|^2_2+|H_{tt}|^2_2).
\end{split}
\end{equation}
Then   multiplying  (\ref{wf3}) by $t$ and integrating the result inequality  over $(\tau,t)$ ($\tau\in (0,t)$), from (\ref{wkui}) and (\ref{wzhen10}),
we have
\begin{equation}
\label{we2}
\begin{split}
&\int_{\tau}^{t}s|\sqrt{\rho}u_{tt}(s)|^2_2\text{d}s+t|\nabla u_t(t)|^2_{2}
\leq \tau | u_t(\tau)|^2_{D^1_0}+C\int_{\tau}^{t}s(1+|\nabla u_t|^2_2)|\nabla u_t|^2_{2}\text{d}s+C
\end{split}
\end{equation}
for $\tau\leq t \leq T$.  From Lemma \ref{wlem:4-1}, we have $\nabla u_{t}\in L^2([0,T];L^2)$, then  according to Lemma \ref{bei},  there exists a sequence ${s_k}$ such that
$$
s_k\rightarrow 0,\quad \text{and} \quad s_k|\nabla u_{t}(s_k)|^2_2\rightarrow 0, \quad \text{as} \quad k\rightarrow \infty.
$$
Therefore, letting $\tau=s_k\rightarrow 0$ in (\ref{we2}),  from Gronwall's inequality, we have
\begin{equation*}
\begin{split}
&\int_{0}^{t}s|\sqrt{\rho}u_{tt}(s)|^2_2\text{d}s+t| u_t(t)|^2_{D^1_0}
\leq C\exp\Big(\int_{0}^{t}(1+|\nabla u_t|^2_2)\text{d}s\Big)\leq C.
\end{split}
\end{equation*}
From (\ref{zhu150}) (\ref{we2}), Lemmas \ref{ok} and \ref{s2}-\ref{s8}
 we immediately have
\begin{equation*}\begin{split}
t|u(t)|^2_{D^3}+\int_0^t s|u_t|^2_{D^2} \text{d}s \leq C(t |u_t(t)|_{D^1_0}+1)+C \int_0^t s(1+|\sqrt{\rho}u_{tt}|^2_2)\text{d}s\leq C.
\end{split}
\end{equation*}

\end{proof}

\begin{lemma}[\textbf{Higher order estimate of the velocity $u$}]\label{wlem:4-4}\ \\
\begin{equation*}
\begin{split}
\big(|(\rho,P,H)(t)|_{D^{2,q}}+t|(\rho_t,P_t,H_t)(t)|_{D^{1,q}}\big)+\int_{0}^{T} |u|^{p_0}_{D^{3,q}}\text{d}t\leq C,
\end{split}
\end{equation*}
where $C$ only depends on $C_0$ and $T$ $(any \ T \in (0,\overline{T}])$.
 \end{lemma}
 \begin{proof}
From  Lemmas \ref{ok} and  \ref{s2}-\ref{wlem:4-3}, we easily obtain
\begin{equation}\label{zhen12}\begin{split}
|u|_{D^{3,q}}
\leq& C\big(|\rho u_t+\rho u\cdot\nabla u|_{D^{1,q}}+|\text{rot}H\times H|_{D^{1,q}}+|P|_{D^{2,q}}\big)\\
\leq & C(|u_t|_\infty+|\nabla u_t|_q+|u|_{D^{2,q}}+|H|_{D^{2,q}}+|P|_{D^{2,q}}).
\end{split}
\end{equation}
Due to the Sobolev inequality, Poincare inequality and Young's inequality, we have
\begin{equation*}\begin{split}
|u_t|_\infty\leq& C|u_t|^{1-\frac{3}{q}}_{q}\| u_t\|^{\frac{3}{q}}_{W^{1,q}}\leq C|\nabla u_t|_2 +C|\nabla u_t|_q,
\end{split}
\end{equation*}
then we have
\begin{equation*}\begin{split}
|u(t)|_{D^{3,q}}
\leq & C(|\nabla u_t|_2+|\nabla u_t|_q+|u|_{D^{2,q}}+|H|_{D^{2,q}}+|P|_{D^{2,q}}).
\end{split}
\end{equation*}
According to Lemmas \ref{s4}-\ref{wlem:4-3}, via the completely same argument in (\ref{moufan}), we have
\begin{equation}\label{guanj}
\begin{split}
&\int_{0}^{t}C(|\nabla u_t|_2+|\nabla u_t|_q+|u|_{D^{2,q}})^{p_0}\text{d}s
\leq C.
\end{split}
\end{equation}

Then,
 applying $\nabla^2$  to  $(\ref{eq:1.2})_3$, and 
multiplying the result equation by $q\nabla^2 \rho|\nabla^2 \rho^{q-2}|$, integrating  over $\Omega$ we easily deduce that
\begin{equation}\label{zhenzhen1d1}\begin{split}
\frac{d}{dt}|\rho|^q_{D^{2,q}}\leq& C|\nabla u|_\infty|\rho|^q_{D^{2,q}}+C|\rho|_{\infty}|u|_{D^{3,q}}|\rho|^{q-1}_{D^{2,q}}+|\nabla \rho|_\infty |u|_{D^{2,q}} |\rho|^{q-1}_{D^{2,q}},
\end{split}
\end{equation}
which, together with (\ref{zhu150}),
\begin{equation}\label{zhenzhenzxc}\begin{split}
\frac{d}{dt}|\rho|_{D^{2,q}}\leq& C(|\nabla u|_\infty+1+F)(1+|\rho|_{D^{2,q}}+|P|_{D^{2,q}}+|H|_{D^{2,q}})+CF,
\end{split}
\end{equation}
where $F=|\nabla u_t|_2+|\nabla u_t|_q+|u|_{D^{2,q}}$. And similarly, we have
\begin{equation}\label{zhenzhenvvv}\begin{cases}
\displaystyle
\frac{d}{dt}|H|_{D^{2,q}}\leq C(|\nabla u|_\infty+F+1)(1+|\rho|_{D^{2,q}}+|P|_{D^{2,q}}+|H|_{D^{2,q}})+F,\\[8pt]
\displaystyle
\frac{d}{dt}|P|_{D^{2,q}}\leq C(|\nabla u|_\infty+F+1)(1+|\rho|_{D^{2,q}}+|P|_{D^{2,q}}+|H|_{D^{2,q}})+F.
\end{cases}
\end{equation}

So we combining (\ref{zhenzhenzxc})- (\ref{zhenzhenvvv}), we quickly have
\begin{equation}\label{wwzhmm12acc}
\begin{split}
&\frac{d}{dt}(|\rho|_{D^{2,q}}+|H|_{D^{2,q}}+|P|_{D^{2,q}})\\
\leq& C(1+|\nabla u|_\infty+F)(1+|\rho|_{D^{2,q}}+|P|_{D^{2,q}}+|H|_{D^{2,q}})+C(1+F).\end{split}
\end{equation}
Then via Gronwall's inequality,   (\ref{guanj}) and (\ref{wwzhmm12acc}),  we obtain
\begin{equation*}
\begin{split}
|\rho|_{D^{2,q}}+|H|_{D^{2,q}}+|P|_{D^{2,q}}+\int_{0}^{t}|u(s)|^{p_0}_{D^{3,q}}\text{d}t\leq C, \quad 0\leq t\leq T.
\end{split}
\end{equation*}

Finally, due to  relation (\ref{ghtuss}),
we immediately get the desired conclusions.

\end{proof}

Finally, we have
\begin{lemma}[\textbf{Higher order estimate of the velocity $u$}]\label{lem:4-4a}
\begin{equation*}
\begin{split}
t^2|u(t)|_{D^{3,q}}+t^2|u_t(t)|^2_{ D^2}+t^2|\sqrt{\rho}u_{tt}(t)|^2_2+\int_{0}^{T}s^2|u_{tt}(s)|^2_{D^1_0}\text{d}s\leq C
\end{split}
\end{equation*}
where $C$ only depends on $C_0$ and $T$ $(any \ T \in (0,\overline{T}])$.
 \end{lemma}
This lemma can be easily proved via the method used in Lemma  \ref{wlem:4-3}, here we omit it.
And this will be enough to extend the regular solutions of $(H,\rho,u,P)$ beyond $t\geq \overline{T}$.

In truth, in view of the estimates obtained in  Lemmas \ref{s2}-\ref{wlem:4-4}, we quickly know that the functions $(H,\rho,u,P)|_{t=\overline{T}} =\lim_{t\rightarrow \overline{T}}(H,\rho,u,P)$ satisfies the conditions imposed on the initial data $(\ref{th78})-(\ref{th79})$. Therefore, we can take $(H,\rho,u,P)|_{t=\overline{T}}$ as the initial data and apply the local existence Theorem \ref{th5} to extend our local classical solution beyond $t\geq \overline{T}$. This contradicts the assumption on $\overline{T}$.

\end{document}